\documentclass[reqno]{amsart}

\usepackage{amsfonts}
\usepackage[all]{xy}
\usepackage{amssymb}
\usepackage{amsmath}
\usepackage{epsfig}
\usepackage{amscd}
\usepackage{graphicx, color}
\usepackage{epstopdf}
\usepackage[all]{xy}
\usepackage{mathrsfs}

\def\E{\ifmmode{\mathbb E}\else{$\mathbb E$}\fi} 
\def\N{\ifmmode{\mathbb N}\else{$\mathbb N$}\fi} 
\def\R{\ifmmode{\mathbb R}\else{$\mathbb R$}\fi} 
\def\Q{\ifmmode{\mathbb Q}\else{$\mathbb Q$}\fi} 
\def\C{\ifmmode{\mathbb C}\else{$\mathbb C$}\fi} 
\def\H{\ifmmode{\mathbb H}\else{$\mathbb H$}\fi} 
\def\Z{\ifmmode{\mathbb Z}\else{$\mathbb Z$}\fi} 
\def\P{\ifmmode{\mathbb P}\else{$\mathbb P$}\fi} 
\def\T{\ifmmode{\mathbb T}\else{$\mathbb T$}\fi} 
\def\SS{\ifmmode{\mathbb S}\else{$\mathbb S$}\fi} 
\def\DD{\ifmmode{\mathbb D}\else{$\mathbb D$}\fi} 

\newcommand{\e}{\varepsilon}

\newcommand{\del}{\partial}

\newcommand{\ben}{\begin{enumerate}}
\newcommand{\een}{\end{enumerate}}
\newcommand{\be}{\begin{equation}}
\newcommand{\ee}{\end{equation}}
\newcommand{\bea}{\begin{eqnarray}}
\newcommand{\eea}{\end{eqnarray}}
\newcommand{\beastar}{\begin{eqnarray*}}
\newcommand{\eeastar}{\end{eqnarray*}}
\newcommand{\bc}{\begin{center}}
\newcommand{\ec}{\end{center}}

\theoremstyle{theorem}
\newtheorem{thm}{Theorem}[section]
\newtheorem{cor}[thm]{Corollary}
\newtheorem{lem}[thm]{Lemma}
\newtheorem{prop}[thm]{Proposition}

\newtheorem{conj}[thm]{Conjecture}

\theoremstyle{definition}
\newtheorem{defn}[thm]{Definition}
\newtheorem{rem}[thm]{Remark}

\newtheorem*{thm*}{Theorem}

\numberwithin{equation}{section}

\def\R{{\mathbb R}}
\def\osc{{\hbox{\rm osc}}}

\def\tr{{\hbox{\bf tr}}}
\def\Tr{{\hbox{\bf Tr}}}

\def\E{{\mathbb E}}
\def\Z{{\mathbb Z}}
\def\C{{\mathbb C}}
\def\R{{\mathbb R}}
\def\P{{\mathbb P}}

\def\N{{\mathbb N}}

\def\11{{\mathbb I}}
\def\K{\mathbb{K}}
\def\H{\mathbb{H}}

\def\C{\mathbb{C}}
\def\Z{\mathbb{Z}}

\def\T{\mathbb{T}}

\def\Q{\mathbb{Q}}

\def\E{\ifmmode{\mathbb E}\else{$\mathbb E$}\fi} 
\def\N{\ifmmode{\mathbb N}\else{$\mathbb N$}\fi} 
\def\R{\ifmmode{\mathbb R}\else{$\mathbb R$}\fi} 
\def\Q{\ifmmode{\mathbb Q}\else{$\mathbb Q$}\fi} 
\def\C{\ifmmode{\mathbb C}\else{$\mathbb C$}\fi} 
\def\Z{\ifmmode{\mathbb Z}\else{$\mathbb Z$}\fi} 
\def\P{\ifmmode{\mathbb P}\else{$\mathbb P$}\fi} 
\def\CS{\ifmmode{\mathbb S}\else{$\mathbb S$}\fi} 
\def\DD{\ifmmode{\mathbb D}\else{$\mathbb D$}\fi} 

\def\R{{\mathbb R}}
\def\osc{{\hbox{\rm osc}}}

\def\E{{\mathbb E}}
\def\Z{{\mathbb Z}}
\def\C{{\mathbb C}}
\def\R{{\mathbb R}}

\def\N{{\mathbb N}}

\def\e{\varepsilon}

\def\CA{{\mathcal A}}

\def\CC{{\mathcal C}}

\def\CG{{\mathcal G}}
\def\CH{{\mathcal H}}

\def\CP{{\mathcal P}}

\def\CP{{\mathcal P}}
\def\CS{{\mathcal S}}

\def\CU{{\mathcal U}}
\def\CV{{\mathcal V}}

\def\darr#1{\raise1.5ex\hbox{$\leftrightarrow$}
\mkern-16.5mu #1}

\def\roughly#1{\raise.3ex\hbox{$#1$\kern-.75em
\lower1ex\hbox{$\sim$}}}

\def\opname#1{\mathop{\kern0pt{\rm #1}}\nolimits}

\def\dim{\opname{dim}}
\def\vol{\opname{vol}}

\def\dudtau{\frac{\partial u}{\partial \tau}}
\def\dudt{\frac{\partial u}{\partial t}}

\def\supp{\operatorname{supp}}
\def\Cal{\operatorname{Cal}}

\def\Graph{\operatorname{Graph}}
\def\Dev{\operatorname{Dev}}

\def\leng{\operatorname{leng}}

\def\Image{\operatorname{Image}}
\def\Int{\operatorname{Int}}

\begin{document}
\quad \vskip1.375truein

\def\mq{\mathfrak{q}}
\def\mH{\mathfrak{H}}
\def\mh{\mathfrak{h}}
\def\ma{\mathfrak{a}}
\def\ms{\mathfrak{s}}
\def\mm{\mathfrak{m}}
\def\mn{\mathfrak{n}}

\def\Hoch{{\tt Hoch}}
\def\mt{\mathfrak{t}}
\def\ml{\mathfrak{l}}
\def\mT{\mathfrak{T}}
\def\mL{\mathfrak{L}}
\def\mg{\mathfrak{g}}
\def\md{\mathfrak{d}}

\title[Area-preserving homeomophism group of $D^2$]{Continuous Hamiltonian dynamics and
area-preserving homeomorphism group of $D^2$}
\author{Yong-Geun Oh}

\address{Center for Geometry and Physics, Institute for Basic Sciences (IBS), Pohang, Korea
77 Cheongam-ro, Nam-gu, Pohang-si, Gyeongsangbuk-do, Korea 790-784 \newline
\& Pohang University of Science and Technology (POSTECH), Pohang, Korea}
\email{yongoh1@postech.ac.kr}

\begin{abstract} The main purpose of this paper is to propose a scheme of a proof of the
nonsimpleness of the group $Homeo^\Omega(D^2,\del D^2)$ of area preserving homeomorphisms
of the 2-disc $D^2$. We first establish the existence of Alexander isotopy in the
category of Hamiltonian homeomorphisms.
This reduces the question of extendability of the well-known Calabi homomorphism
$\Cal: Diff^\Omega(D^1,\del D^2) \to \R$ to a homomorphism $\overline \Cal: Hameo(D^2,\del D^2) \to \R$
to that of the vanishing of the basic phase function $f_{\underline{\mathbb F}}$,
a Floer theoretic graph selector constructed in \cite{oh:jdg}, that is associated to
the graph of the topological Hamiltonian loop and its normalized Hamiltonian $\underline{F}$
on $S^2$ that is obtained via the natural embedding $D^2 \hookrightarrow S^2$.
Here $Hameo(D^2,\del D^2)$ is the group of
Hamiltonian homeomorphisms introduced by M\"uller and the author \cite{oh:hameo1}.
We then provide an evidence of this vanishing conjecture
by proving the conjecture for the special class of \emph{weakly graphical}
topological Hamiltonian loops on $D^2$.
\end{abstract}


\thanks{The present work is supported by the IBS project \# IBS-R003-D1}

\date{January 2015; revised in September 2015}

\maketitle

\medskip

\tableofcontents

\section{Introduction and statements of main results}
\label{sec:vanishing}

\subsection{Calabi homomorphism on $D^2$}

Denote by $Diff^\Omega(D^2, \del D^2)$ the group of area-preserving diffeomorphisms
supported in the interior of $D^2$ with respect to the standard area form
$\Omega = dq\wedge dp$ on $D^2 \subset \R^2$. For any $\phi \in Diff^\Omega(D^2,\del D^2)$
$$
\phi^*\Omega = \Omega
$$
by definition. Write $\Omega = d\alpha$ for some choice of
$\alpha$. Then this equation leads to the statement $\phi^*\alpha - \alpha$ is
closed. Furthermore since $\phi$ is supported in the interior, the one-form
$$
\phi^*\alpha - \alpha
$$
vanishes near $\del D^2$ and so defines a de Rham cohomology class
lying in $H^1(D^2,\del D^2)$. Since the latter group is trivial, we can find
a function $h_{\phi,\alpha}$ supported in the interior such that
\be\label{eq:dhphi}
dh_{\phi,\alpha} = \phi^*\alpha - \alpha.
\ee
Then the following is the well-known definition of Calabi invariant \cite{calabi}.
\begin{defn}[Calabi invariant]\label{defn:Cal-defn}
We define
$$
\Cal(\phi) = \frac{1}{2} \int_{D^2} h_{\phi,\alpha}.
$$
\end{defn}
One can show that this value does not depend on the choice of
the one-form $\alpha$ but depends only on the diffeomorphism.
We will fix one such form $\alpha$ and so suppress the dependence
$\alpha$ from our notation, and just denote $h_\phi = h_{\phi,\alpha}$.

Another equivalent definition does not involve the choice of one-form
$\alpha$ but uses the `past history' of the diffeomorphism in the setting of
Hamiltonian dynamics \cite{banyaga}. More precisely, this definition implicitly relies on the
following three facts:
\begin{enumerate}
\item $\Omega$ on two dimensional surface is a symplectic form and hence
$$
Diff^\Omega(D^2,\del D^2) = Symp_\omega(D^2,\del D^2)
$$
where $\omega = \Omega$.
\item $D^2$ is simply connected, which
in turn implies that any symplectic isotopy is a Hamiltonian isotopy.
\item The group $Diff^\Omega(D^2,\del D^2)$ is contractible.
(For this matter, finiteness of $\pi_1(Diff^\Omega(D^2,\del D^2),id) \cong \{0\}$ is enough.)
\end{enumerate}

It is well-known (see \cite{gamb-ghys}, \cite{oh:simpleD2} for example) and easy to construct a sequence
$\phi_i \in Diff^\Omega(D^2,\del D^2)$ such that $\phi_i \to id$ in $C^0$ topology but
$$
\Cal(\phi_i) = 1
$$
for all $i$'s. This implies that $\Cal$ cannot be continuously extended to the full group
$Homeo^\Omega(D^2,\del D^2)$ of area-preserving homeomorphisms.

However here is the main conjecture of the paper concerning nonsimpleness of
the group $Homeo^\Omega(D^2,\del D^2)$. The author learned from A. Fathi
in our discussion on the group $Hameo(D^2,\del D^2)$ \cite{fathi-private} that
the following will be important in relation to the study of nonsimpleness
conjecture. (We refer to \cite{oh:hameo2} for the argument needed to complete this nonsimpleness proof
out of this conjecture.)
\medskip

\begin{conj}\label{conj:Calabiextend}
Let $Hameo(D^2,\del D^2) \subset Homeo^\Omega(D^2,\del D^2)$ be the subgroup
of Hamiltonian homeomorphisms on the two-disc.
Then the Calabi homomorphism $\Cal: Diff^\Omega(D^2,\del D^2) \to \R$ extends
continuously to $Hameo(D^2,\del D^2)$ in Hamiltonian topology in the sense
of \cite{oh:hameo1}.
\end{conj}

For the study towards this conjecture, as we did in \cite{oh:hameo2}, we first define
a homomorphism on the path spaces
$$
\Cal^{path}(\lambda): \CP^{ham}(Symp(D^2,\del D^2),id) \to \R
$$
by
\be\label{eq:Calpath}
\Cal^{path}(\lambda) = \int_0^1 \int_{D^2} H(t,x)\, \Omega\, dt.
\ee
We will also denote this average by $\Cal(H)$ depending on the circumstances.
Based on these facts (1) and (2), we can represent $\phi= \phi_H^1$
for the time-one map $\phi_H^1$ of
a time-dependent Hamiltonian $H = H(t,x)$ supported in the interior.
Then based on (3) and some standard calculations in Hamiltonian
geometry using the integration by parts, one proves that
this integral does not depend on the choice of Hamiltonian $H \mapsto \phi$.
Therefore it descends to $Ham(D^2,\del D^2) = Diff^\Omega(D^2,\del D^2)$. Then another
application of Stokes' formula, one can prove that this latter definition
indeed coincides with that of Definition \ref{defn:Cal-defn}. (See \cite{banyaga} for its
proof.)

It is via this second definition how the author attempts to extend the classical
Calabi homomorphism $\Cal: Ham(D^2,\del D^2) \to \R$
to its topological analog $\overline\Cal: Hameo(D^2,\del D^2) \to \R$.
In \cite{oh:hameo2}, the definition \eqref{eq:Calpath} is extended to a homomorphism
$$
\overline{\Cal}^{path}:\CP^{ham}(Sympeo(D^2,\del D^2),id) \to \R
$$
on the set $\CP^{ham}(Sympeo(D^2,\del D^2),id)$ of topological Hamiltonian paths.
(See section 2 for the precise definition.)
Here following the notation from \cite{oh:hameo1},
we denote by $Sympeo(D^2,\del D^2)$ the $C^0$-closure of $Symp(D^2,\del D^2)$.
Gromov-Eliashberg's $C^0$ symplectic rigidity theorem \cite{eliash:front} states
$$
Diff(D^2,\del D^2) \cap Sympeo(D^2,\del D^2) = Symp(D^2,\del D^2).
$$
In \cite{oh:hameo2,oh:simpleD2}, it is shown that
a proof of descent of $\overline{\Cal}^{path}$ to the group
$$
Hameo(D^2,\del D^2):= ev_1(\CP^{ham}(Sympeo(M,\omega),id))
$$
of Hamiltonian homeomorphisms (or more succinctly \emph{hameomorphisms})
is reduced to the following extension result of Calabi homomorphism.

One important ingredient in our scheme towards the proof of Conjecture \ref{conj:Calabiextend},
which itself has its own interest, is the existence of
the Alexander isotopy in the topological Hamiltonian category.
Recall that the well-known Alexander isotopy on the disc $D^2$ exists
in the homeomorphism category but not
in the differentiable category. We will establish that such an Alexander isotopy
defines contractions of topological Hamiltonian loops to the
identity constant loop in the topological Hamiltonian category.

\begin{thm}[Alexander isotopy; Theorem \ref{thm:alexander}]\label{thm:alex}
Any topological Hamiltonian loop in $Hameo(D^2,\del D^2)$ is
contractible to the identity loop via topological Hamiltonian homotopy of loops.
\end{thm}

\subsection{Basic phase function and Calabi invariant}
\label{subsec:basic-Calabi}

The scheme of the proof of Conjecture \ref{conj:Calabiextend} we propose
is based on the following conjectural result of the \emph{basic phase function}
introduced in \cite{oh:jdg}.  This conjecture is also a crucial ingredient
needed in the proof of homotopy invariance of the spectral invariance of
topological Hamiltonian paths laid out in \cite{oh:homotopy}.
Explanation of this conjecture is now in order.

Recall the classical action functional on $T^*N$ for an arbitrary compact manifold $N$ is defined as
$$
\CA^{cl}_H(\gamma) = \int \gamma^*\theta - \int_0^1 H(t,\gamma(t))\, dt
$$
on the space $\CP(T^*N)$ of paths $\gamma:[0,1] \to T^*N$, and its first variation formula
is given by
\be\label{eq:1stvariation}
d\CA^{cl}_H(\gamma)(\xi) = \int_0^1 \omega(\dot\gamma - X_H(t,\gamma(t)), \xi(t))\,dt
- \langle \theta(\gamma(0)),\xi(0) \rangle + \langle \theta(\gamma(1)), \xi(1) \rangle.
\ee
The basic phase function graph selector is canonical in that the assignment
$$
H \mapsto f_H; \quad C^\infty([0,1] \times T^*N;\R) \to C^0(N)
$$
varies continuously in (weak) Hamiltonian topology of $C^\infty([0,1] \times T^*N;\R)$
\cite{oh:hameo1,oh:lag-spectral}.
The construction $f_H$ in \cite{oh:jdg} is given by considering the Lagrangian pair
$$
(o_N, T^*_qN), \quad q \in N
$$
and its associated Floer complex $CF(H;o_N, T^*_qN)$ generated by
the Hamiltonian trajectory $z:[0,1] \to T^*N$ satisfying
\be\label{eq:Hamchordeq1}
\dot z = X_H(t,z(t)), \quad z(0) \in o_N, \, z(1) \in T^*_qN.
\ee
Denote by $\CC hord(H;o_N,T^*_qN)$ the set of
solutions of \eqref{eq:Hamchordeq1}. The differential $\del_{(H,J)}$ on $CF(H;o_N, T^*_qN)$ is
provided by the moduli space of solutions of the perturbed Cauchy-Riemann equation
\be\label{eq:CRHJqN}
\begin{cases}
\dudtau + J\left(\dudt - X_H(u) \right) = 0 \\
u(\tau,0) \in o_N, \, u(\tau,1) \in T^*_qN.
\end{cases}
\ee
The resulting spectral invariant $\rho^{lag}(H;[q])$ is to be defined by the mini-max
value
$$
\rho^{lag}(H;[q]) = \inf_{\alpha \in [q]}\lambda_H(\alpha)
$$
where $[q]$ is a generator of the homology group $HF(o_N, T^*_qN) \cong \Z$.
The basic phase function $f_H: N \to \R$ is then defined by $f_H(q) = \rho^{lag}(H;[q])$ first for
generic $q \in N$ and then extending to the rest of $M$ by continuity.
(See \cite{oh:jdg} for the detailed construction and section \ref{sec:basicphase} of the
present paper for a summary.)

Next we relate the basic phase function to the Calabi invariant on the two-disc as follows.
Let $F$ be a topological Hamiltonian
generating a topological Hamiltonian path $\phi_F$ on the
2-disc $D^2$ with $\supp F \subset \Int D^2$.
We consider an approximating sequence $F_i$
with $\supp F_i \subset \Int D^2$. We embed $D^2$ into $S^2$ as the upper
hemisphere and then extend $F_i$ canonically to whole $S^2$ by zero.

We now specialize the above discussion on the basic phase function to the cases of the Lagrangianization of symplectic
diffeomorphisms, i.e., consider their graphs
$$
\Graph \phi=\{(\phi(x),x) \mid x \in S^2\} \subset S^2 \times S^2.
$$
Applying this to $\phi_{\mathbb F_i}^t$ and noting $\supp \phi_{\mathbb F_i}^t \subset D^2_+ \times D^2_+$,
we obtain
$$
\Graph \phi_{F_i}^t \bigcap \Delta \supset \Delta_{D^2_-} \bigcup \Delta_{D^2_+ \setminus D^2_+(1-\delta)}
$$
for some $\delta > 0$ for all $t \in [0,1]$, independently of sufficiently large $i$'s but depending only on $F$.
(See \cite{oh:hameo1} or Definition \ref{defn:directlimit}
of the present paper for the precise definition of approximating sequence on open manifolds.)
Then we consider the normalization $\underline{F_i}$ of $F_i$ on $S^2$ and define Hamiltonian
$$
\underline{\mathbb F_i} (t,{\bf x}): = \chi({\bf x})\, \underline{F_i}(t,x), \quad {\bf x} = (x,y)
$$
on $T^*\Delta$ with a slight abuse of notation for $\underline{\mathbb F_i}$, where
$\chi$ is a cut-off function such that $\chi \equiv 1$ on a neighborhood $V_\Delta$ of $\Delta$
with
$$
\supp \phi_F \subset V_\Delta \subset \overline V_\Delta \subset S^2 \times S^2 \setminus \Delta.
$$
Two kinds of the associated generating functions, denoted by $\widetilde h_{\mathbb F_i}$ and $h_{\mathbb F_i}$ respectively,
are given by
\be\label{eq:tildehK1}
\widetilde h_{\mathbb F_i}({\bf q}) = \CA^{cl}_{\mathbb F}\left(z^{\bf q}_{\mathbb F_i}\right), \quad
h_{\mathbb F_i}({\bf x})=  \CA^{cl}_{\mathbb F_i}\left(z_{\bf x}^{\mathbb F_i}\right),
\ee
where the Hamiltonian trajectories $z^{\bf q}_{\mathbb F_i}$ and $z_{\bf x}^{\mathbb F_i}$ are defined by
\beastar
z^{\bf q}_{\mathbb F_i}(t) &  =  & \phi_{\mathbb F_i}^t({\bf q}), \, {\bf q} \in o_\Delta\\
z_{\bf x}^{\mathbb F_i}(t) & = & \phi_{\mathbb F_i}^t((\phi_{\mathbb F_i}^1)^{-1}({\bf x})), \, {\bf x} \in
\phi_{\mathbb F_i}^1(o_\Delta).
\eeastar
We note that $z^{\bf q}_{\mathbb F_i}(0) = {\bf q}$ and $z_{\bf x}^{\mathbb F_i}(1) = {\bf x}$.
Later we will review the definition from \cite{oh:jdg,oh:lag-spectral} of
the basic phase function $f_{\mathbb F_i}$ and the Lagrangian selector $\sigma_{\mathbb F_i}$.
These maps have the properties that
\be\label{eq:fK1}
f_{\mathbb F_i}  = h_{\mathbb F_i}\circ \sigma_{\mathbb F_i}
\ee
and
$
\sigma_{\mathbb F_i}({\bf q}) = ({\bf q}, df_{\mathbb F_i}({\bf q})) \in T^*\Delta
$
whenever $df_{\mathbb F_i}({\bf q})$ exists. This ends the review of construction of
basic phase function.

The following theorem exhibits the relationship between the limit of
Calabi invariants and that of the basic phase function.

\begin{thm}[Theorem \ref{thm:average=Cal}]\label{thm:fFi=Cal} Let $(M,\omega)$ be an
arbitrary closed symplectic manifold. Let $U = M \setminus B$ where $B$ is a closed subset of nonempty interior.
Let $\lambda = \phi_F$ be any engulfed topological Hamiltonian loop
in $\CP^{ham}(Sympeo_U(M,\omega),id)$ with $\phi_{F_i}^t \equiv id$ on $B$.  Then
\be\label{eq:limitfHi}
\lim_{i \to \infty} f_{\underline{\mathbb F_i}}(x) =  \frac{\overline{\Cal}_U(F)}{\vol_\omega(M)}
\ee
uniformly over $x \in M$, for any approximating sequence $F_i$ of $F$. In particular,
the limit function $f_{\underline{\mathbb F}}$ defined by
$f_{\underline{\mathbb F}}(x): = \lim_{i \to \infty} f_{\underline{\mathbb F_i}}(x)$ is constant.
\end{thm}

It is crucial for the equality \eqref{eq:limitfHi} to hold in the general case
that we are considering topological Hamiltonian \emph{loop}, not just a path.
(We refer readers to the proof of Theorem \ref{thm:average=Cal} to see how the loop property
is used therein. We also refer to the proof of Lemma 7.5 \cite{oh:lag-spectral} for
a similar argument used for a similar purpose.)

The following is the main conjecture to beat which was previously proposed
by the present author in \cite{oh:homotopy}.

\begin{conj}[Main Conjecture]\label{conj:fvanishing} Let $M = S^2$ be the 2 sphere with standard symplectic structure.
Let $\Lambda= \left\{\phi_{H(s)}^t\right\}_{(s,t) \in [0,1]^2}$ be a hameotopy
contracting a topological Hamiltonian loop
$\phi_F$ with $F = H(1)$ such that $H(s) \equiv id$ on $D^2_-$ where $D^2_-$ is
the lower hemisphere of $S^2$. Then
$
f_{\underline{\mathbb F}} = 0.
$
\end{conj}

It turns out that this conjecture itself is strong enough to directly give rise to
a proof of Conjecture \ref{conj:Calabiextend}
in a rather straightforward manner with little usage of Floer homology argument
in its outset except a few functorial properties of the basic phase function that are
automatically carried by the Floer theoretic construction given in \cite{oh:jdg}.

We indicate validity of this conjecture by proving the conjecture for the following special
class consisting of weakly graphical topological Hamiltonian loops.

\subsection{Graphical Hamiltonian diffeomorphism on $D^2$ and its Calabi invariant}

We start with the following definition. We refer readers to Definition \ref{defn:engulfed}
for the definition of engulfed diffeomorphisms.

\begin{defn} Let $\Psi: U_\Delta \to \CV$ be a Darboux-Weinstein chart of the diagonal $\Delta \subset M \times M$
and denote $\pi_\Delta = \pi_\Delta^\Psi: U_\Delta \to \Delta$ to be the composition of $\Psi$ followed by the
canonical projection $T^*\Delta \to \Delta$.
\begin{enumerate}
\item We call an engulfed symplectic diffeomorphism $\phi: M \to M$ \emph{$\Psi$-graphical} if the projection
$\pi_\Delta|_{\Graph \phi} \to \Delta$ is one-to-one, and an engulfed symplectic isotopy
is $\{\phi^t\}$ \emph{$\Psi$-graphical} if
each element $\phi^t$ $\Psi$-graphical. We call a Hamiltonian $F=F(t,x)$ $\Psi$-graphical if
its associated Hamiltonian isotopy $\phi_F^t$ is $\Psi$-graphical.
\item We call a topological Hamiltonian loop $F$ is strongly (resp. weakly) $\Psi$-graphical, if it admits an
approximating sequence $F_i$ each element of which is $\Psi$-graphical (resp. whose time-one map $\phi_{F_i}^1$
is $\Psi$-graphical).
\end{enumerate}
\end{defn}
Denote by $F^a$ the time-dependent Hamiltonian generating the
path $t \mapsto \phi_F^{at}$. The statement (2) of this definition is
equivalent to saying that each $F^a$ is $\Psi$-graphical for $a \in [0,1]$.

We remark that any symplectic diffeomorphisms sufficiently $C^1$-close to the identity
is graphical, but not every $C^0$-close one. We also remark that
$\pi_\Delta|_{\Graph \phi}$ is surjective and hence a diffeomorphism
if $\phi$ is a $\Psi$-graphical symplectic diffeomorphism isotopic to the identity
via a $\Psi$-engulfed isotopy.

In 2 dimension, we prove the following interesting phenomenon.
We doubt that similar phenomenon occurs in high dimension. This theorem
will not be used in the proofs of main results of the present paper but
has its own interest.

\begin{thm}\label{thm:onetoone} Let $M$ be a closed 2 dimensional surface.
Suppose $\phi: M \to M$ is a $\Psi$-graphical symplectic diffeomorphism
isotopic to the identity via $\Psi$-graphical isotopy.
and let $\Graph \phi = \Image \alpha_\phi$ for a closed one-form $\alpha_\phi$.
Then for any $0 \leq r \leq 1$,
the projection $\pi_2: M \times M \to M$ restricts to a one-one map
to $ \Image r\, \alpha_\phi \subset M \times M$. In particular
\be\label{eq:tdfphit}
\Image r\, \alpha_\phi = \Graph \phi_r
\ee
for some symplectic diffeomorphism $\phi_r: M \to M$ for each $ 0 \leq r \leq 1$.
\end{thm}

Finally we prove Conjecture \ref{conj:fvanishing} for the weakly graphical topological
Hamiltonian loop on $S^2$ that arises as follows.

\begin{thm}\label{thm:weak-graphical}
Conjecture \ref{conj:fvanishing} holds for any weakly graphical topological
Hamiltonian loop on $S^2$ arising from one on $D^2$
as in subsection \ref{subsec:basic-Calabi}.
\end{thm}
The proof of this theorem strongly relies on Theorem \ref{thm:alex}.

An immediate corollary of Theorem \ref{thm:fFi=Cal} and \ref{thm:weak-graphical}
is the following vanishing result of Calabi invariant.

\begin{cor}
Suppose $\lambda = \phi_F$ is a weakly graphical topological Hamiltonian
loop on $D^2$. Then $\overline{\Cal}^{path}(\lambda) = 0$.
\end{cor}

Unraveling the definitions, this corollary establishes the main conjecture
with the additional graphicality hypothesis on $\phi_i$.

\begin{thm} Consider a sequence $\phi_i \in Ham(D^2,\del D^2)$ that satisfies the following conditions:
\begin{enumerate}
\item Each $\phi_i$ is graphical, and $\phi_i \to id$ in $C^0$-topology,
\item $\phi_i = \phi_{H_i}^1$ with convergent $H_i$ in $L^{(1,\infty)}$-topology.
\end{enumerate}
Then $\lim_{i \to \infty} \Cal(\phi_i) = 0$.
\end{thm}

We hope to study elsewhere general engulfed topological
Hamiltonian loop dropping the graphicality condition.
\begin{rem}
Previously the author announced a `proof' of the nonsimpleness result in \cite{oh:simpleD2} \emph{modulo
the proof of Conjecture \ref{conj:fvanishing}} in which nonsimpleness is derived out of
the homotopy invariance of spectral invariants whose proof also strongly relied
on this vanishing result. Unlike the previously proposed scheme of the proof,
the current scheme does not rely on the homotopy invariance of spectral invariants
of topological Hamiltonian paths but more directly follows from the
above mentioned vanishing result.
\end{rem}

We thank M. Usher for his careful reading of the previous version of the
present paper and useful discussions in relation to the proof of Theorem \ref{thm:graphicalloop}.
We also take this opportunity to thank A. Fathi for explaining to us, during his visit
of KIAS in the summer of year 2004, how the question of extendability of the Calabi homomorphism on $Diff^\Omega(D^2,\del D^2)$
to $Hameo(D^2,\del D^2)$ is related to the non-simpleness of the area-preserving homeomorphism
group $Homeo^\Omega(D^2,\del D^2)$. We also thank much the anonymous referee for
pointing out an error in the proof of Theorem \ref{thm:graphicalloop}. We
newly add section \ref{sec:weak-alexander} in the present version to fix the error,
and revise the proof of Theorem \ref{thm:graphicalloop} accordingly.
We also thank her/him for making
many useful comments which we believe greatly improve our exposition of the paper.

\part{Calabi invariant and basic phase function}

\section{Calabi homomorphism $\overline{\Cal}^{path}$ on the path space}
\label{sec:CalabiD2}

\subsection{Hamiltonian topology and hamiltonian homotopy}
\label{subsec:top-flows}

In \cite{oh:hameo1}, M\"uller and the author introduced the notion of
Hamiltonian topology on the space
$$
\CP^{ham}(Symp(M,\omega),id)
$$
of Hamiltonian paths $\lambda:[0,1] \to Symp(M,\omega)$ with
$\lambda(t) = \phi_H^t$ for some time-dependent Hamiltonian $H$.
We would like to emphasize that we do \emph{not} assume that $H$
is normalized \emph{unless otherwise said explicitly}. This is because we need to
consider both compactly supported and mean-normalized Hamiltonians
and suitably transform one to the other in the course of the proof of
the main theorem of this paper.

We first recall the definition of this Hamiltonian topology.

We start with the case of closed $(M,\omega)$.
For a given continuous function $h: M \to \R$, we denote
$$
\osc(h) = \max h - \min h.
$$
We define the $C^0$-distance $\overline d$ on $Homeo(M)$ by the symmetrized
$C^0$-distance
$$
\overline d(\phi,\psi) = \max\left\{d_{C^0}(\phi,\psi), d_{C^0}(\phi^{-1},\psi^{-1})\right\}
$$
and the $C^0$-distance, again denoted by $\overline d$, on
$$
\CP^{ham}(Symp(M,\omega),id) \subset \CP(Homeo(M),id)
$$
by
$$
\overline d(\lambda,\mu) = \max_{t \in [0,1]} \overline d(\lambda(t),\mu(t)).
$$
The Hofer length of Hamiltonian path $\lambda = \phi_H$ is defined by
$$
\leng(\lambda) = \int_0^1 \osc(H_t)\, dt = \|H\|.
$$
Following the notations of \cite{oh:hameo1}, we denote by $\phi_H$ the
Hamiltonian path
$$
\phi_H: t \mapsto \phi_H^t; \, [0,1] \to Ham(M,\omega)
$$
and by $\Dev(\lambda)$ the associated normalized Hamiltonian
\be\label{eq:Dev}
\Dev(\lambda) := \underline H, \quad \lambda = \phi_H
\ee
where $\underline H$ is defined by
\be\label{eq:underlineH}
\underline H(t,x) = H(t,x) - \frac{1}{\vol_\omega(M)} \int_M H(t,x)\, \omega^n.
\ee

\begin{defn}\label{defn:hamtopology} Let $(M,\omega)$ be a closed symplectic
manifold. Let $\lambda, \, \mu$ be smooth Hamiltonian paths. The
\emph{Hamiltonian topology} is the metric topology induced by the
metric
\be\label{eq:strong} d_{ham}(\lambda,\mu): = \overline
d(\lambda,\mu) + \operatorname{leng}(\lambda^{-1}\mu).
\ee
\end{defn}

Now we recall the notion of topological Hamiltonian flows and
Hamiltonian homeomorphisms introduced in \cite{oh:hameo1}.

\begin{defn}[$L^{(1,\infty)}$ topological Hamiltonian flow]\label{topflowdefn} A continuous map
$\lambda: \R \to Homeo(M)$ is called a topological Hamiltonian flow
if there exists a sequence of smooth Hamiltonians $H_i: \R \times M
\to \R$ satisfying the following:
\begin{enumerate}
\item $\phi_{H_i} \to \lambda$ locally uniformly on $\R \times M$.
\item the sequence $H_i$ is Cauchy in the $L^{(1,\infty)}$-topology locally in time
and so has a limit $H_\infty$ lying in $L^{(1,\infty)}$ on any compact interval $[a,b]$.
\end{enumerate}
We call any such $\phi_{H_i}$ or $H_i$ an \emph{approximating sequence} of $\lambda$.
We call a continuous path $\lambda:[a,b] \to Homeo(M)$ a {\it
topological Hamiltonian path} if it satisfies the same conditions
with $\R$ replaced by $[a,b]$, and the limit $L^{(1,\infty)}$-function
$H_\infty$ called a \emph{$L^{(1,\infty)}$ topological Hamiltonian} or just
a \emph{topological Hamiltonian}.
\end{defn}

Following the notations from \cite{oh:hameo1}, we denote by $Sympeo(M,\omega)$ the
closure of $Symp(M,\omega)$ in $Homeo(M)$ with respect to the $C^0$-metric $\overline d$, and
by $\CH_m([0,1] \times M, \R)$ the set of mean-normalized
topological Hamiltonians, and by
\be\label{eq:ev1}
ev_1: \CP_{[0,1]}^{ham}(Sympeo(M,\omega),id) \to
Sympeo(M,\omega),id)
\ee
the evaluation map defined by
$
ev_1(\lambda) = \lambda(1).
$
By the uniqueness theorem of Buhovsky-Seyfaddini \cite{buh-sey}, we can extend the map
$\operatorname{Dev}$ given in \eqref{eq:Dev} to
$$
\overline{\operatorname{Dev}}: \CP_{[0,1]}^{ham}(Sympeo(M,\omega),id) \to
\CH_m([0,1] \times M, \R)
$$
in an obvious way. Following the notation of \cite{oh:hameo1,oh:hameo2}, we denote
the topological Hamiltonian path $\lambda = \phi_H$ when $\overline{\Dev}(\lambda) = \underline H$
in this general context.

\begin{defn}[Hamiltonian homeomorphism group] We define
$$
Hameo(M,\omega) = ev_1\left(\CP_{[0,1]}^{ham}(Sympeo(M,\omega),id)\right)
$$
and call any element therein a \emph{Hamiltonian homeomorphism}.
\end{defn}

The group property and its normality in $Sympeo(M,\omega)$ are proved in
\cite{oh:hameo1}.

\begin{thm}[\cite{oh:hameo1}]\label{thm:normality}
Let $(M,\omega)$ be a closed symplectic manifold.
Then $Hameo(M,\omega)$ is a normal subgroup of $Sympeo(M,\omega)$.
\end{thm}

Especially when $\dim \Sigma = 2$, we have a smoothing result
\be\label{eq:smoothing}
Sympeo(\Sigma,\omega) = Homeo^\Omega(\Sigma)
\ee
of area-preserving homeomorphisms by area-preserving diffeomorphisms
(see \cite{oh:smoothing}, \cite{sikorav:smoothing} for a proof). Therefore combining this with the above
theorem, we obtain the following corollary, which is the starting point of our
research to apply continuous Hamiltonian dynamics to the study of
the simpleness question of the area-preserving homeomorphism group of $D^2$ (or $S^2$).

\begin{cor} Let $\Sigma$ be a compact surface with or without boundary and
let $\Omega$ be an area form of $\Sigma$, which we also regard as a symplectic form $\omega = \Omega$.
Then $Hameo(M,\omega)$ is a normal subgroup of $Homeo^\Omega(\Sigma)$.
\end{cor}

Both results have their counterparts even when $\del M \neq \emptyset$. We refer to the discussion
below at the end of this subsection.

Next we consider the notion of homotopy in this topological Hamiltonian category.
The following notion of hamiltonian homotopy, which we abbreviate as \emph{hameotopy}, of topological
hamiltonian paths is introduced in \cite{oh:simpleD2,oh:local}. The guiding principle
for the choice of this as the definition of homotopy in this topological Hamiltonian category
is to include the Alexander isotopy we define in section \ref{sec:alexander}
as a special case.

\begin{defn}[Hameotopy]\label{defn:hameotopy} Let $\lambda_0, \, \lambda_1 \in
\CP^{ham}(Sympeo(M,\omega),id)$. A hameotopy
$\Lambda:[0,1]^2 \to Sympeo(M,\omega)$ between
$\lambda_0$ and $\lambda_1$ based at the identity is a map such that
\be\label{eq:ham-homotopy}
\Lambda(0,t) = \lambda_0(t), \, \Lambda(1,t) = \lambda_1(t), \,
\ee and
$\Lambda(s,0) \equiv id$ for all $s \in [0,1]$, and which arises as follows:
there is a sequence of smooth maps $\Lambda_j: [0,1]^2 \to Ham(M,\omega)$ that satisfy
\begin{enumerate}
\item $\Lambda_j(s,0) = id$,
\item $\Lambda_j \to \Lambda$ in $C^0$-topology,
\item Any $s$-section $\Lambda_{j,s}:\{s\} \times [0,1] \to Ham(M,\omega)$
converges in hamiltonian topology in the following sense: If we write
$$
\Dev\left(\Lambda_{j,s}\Lambda_{j,0}^{-1}\right)=: H_j(s),
$$
then $H_j(s)$ converges in hamiltonian topology uniformly over $s \in [0,1]$.
We call any such $\Lambda_j$ an \emph{approximating sequence} of $\Lambda$.
\end{enumerate}

When $\lambda_0(1) = \lambda_1(1) = \psi$, a \emph{hameotopy relative to the ends}
is one that satisfies $\Lambda(s,0) = id, \, \Lambda(s,1) = \psi$ for all $s \in [0,1]$ in addition.

We say that $\lambda_0, \, \lambda_1 \in
\CP^{ham}(Sympeo(M,\omega),id)$ are \emph{hameotopic} (resp. relative to the ends), if there exists a hameotopy
(resp. a hameotopy relative to the ends).
\end{defn}
We emphasize that by the requirement (3),
\be\label{eq:Hj0}
H_j(0) \equiv 0
\ee
in this definition.

All the above definitions can be modified to handle the case of open
manifolds, either noncompact or compact with boundary, by considering
$H$'s compactly supported in the interior as done in section 6 \cite{oh:hameo1}.
We recall the definitions of topological Hamiltonian paths and
Hamiltonian homeomorphisms supported in an open subset $U \subset M$
from \cite{oh:hameo1}.

We first define  $\CP^{ham}(Symp_U(M,\omega),id)$ to be
the set of smooth Hamiltonian paths supported in $U$.
The following definition is taken from Definition 6.2 \cite{oh:hameo1}
to which we refer readers for more detailed discussions.
First for any open subset $V \subset U$ with compact closure $\overline V \subset U$,
we can define a completion of $\CP^{ham}(Symp_{\overline V}(M,\omega),id)$
using the same metric given above.

\begin{defn}\label{defn:directlimit} Let $U \subset M$ be an open subset.
Define $\CP^{ham}(Sympeo_U(M,\omega),id)$ to be the union
$$
\CP^{ham}(Sympeo_U(M,\omega),id) := \bigcup_{K\subset U} \CP^{ham}(Sympeo_K(M,\omega),id)
$$
with the direct limit topology, where $K \subset U$ is a compact subset.
We define $Hameo_c(U,\omega)$ to be the image
$$
Hameo_c(U,\omega): = ev_1(\CP^{ham}(Sympeo_U(M,\omega),id)).
$$
\end{defn}

We would like to emphasize that this set is not necessarily the same as
the set of $\lambda \in \CP^{ham}(Sympeo(M,\omega),id)$
with compact $\supp \lambda \subset U$. The same definition can be
applied to general open manifolds or manifolds with boundary.

\subsection{Calabi invariants of topological Hamiltonian paths in $D^2$}

Denote by $\CP^{ham}(Symp(D^2,\del D^2);id)$ the group of Hamiltonian
paths supported in $\operatorname{Int}(D^2)$, i.e.,
$$
\bigcup_{t \in [0,1]}\supp H_t \subset \operatorname{Int}(D^2).
$$
We denote by $\CP^{ham}(Sympeo(D^2,\del D^2),id)$ the $L^{(1,\infty)}$ hamiltonian
completion of $\CP^{ham}(Symp(D^2,\del D^2);id)$.

We recall the extended Calabi homomorphism defined in \cite{oh:hameo2} whose
well-definedness follows from the uniqueness theorem from \cite{buh-sey}.

\begin{defn} Let $\lambda \in \CP^{ham}(Sympeo(D^2,\del D^2),id)$ and $H$ be
its Hamiltonian supported in $\Int D^2$. We define
$$
\overline{\Cal}^{path}(\lambda)=
\overline{\Cal}^{path}_{D^2}(\lambda) := \overline\Cal(H)
$$
where we define $\overline\Cal(H) = \lim_{i \to \infty} \Cal(H_i)$ for
an (and so any) approximating sequence $H_i$ of $H$.
\end{defn}
It is immediate to check that this defines a homomorphism. The main question
to be answered is whether this homomorphism descends to the group
$Hameo(D^2,\del D^2)$. We recall that one crucial ingredient
needed in the proof of well-definedness of this form of the Calabi invariant
defined on $Diff^\Omega(D^2,\del D^2)$ of area-preserving \emph{diffeomorphisms}
is the fact that $Diff^\Omega(D^2,\del D^2)= Ham(D^2,\del D^2)$ and it is
contractible.
In this regard, we would like to prove the following conjecture.

\begin{conj}\label{conj:loop-vanishing} Let $\lambda$ be a
contractible topological hamiltonian loop based at the identity.
Then
$$
\overline{\Cal}^{path}(\lambda) = 0.
$$
\end{conj}
In the next section, we will establish the existence of Alexander isotopy in the
topological Hamiltonian category and prove that any topological hamiltonian loop (based at the identity)
on $D^2$ is indeed contractible and so the contractibility hypothesis in this conjecture automatically holds.

By the homomorphism property of $\overline{\Cal}^{path}$,
an immediate corollary of this conjecture would be the following:
Suppose that Conjecture \ref{conj:loop-vanishing} holds.
Let
$$
\overline{\Cal}^{path}:\CP^{ham}(Sympeo(D^2,\del D^2),id) \to \R
$$
be the above extension of the Calabi homomorphism $\Cal^{path}$ such that
$\lambda_0(1) = \lambda_1(1)$. Then we have
$$
\overline{\Cal}^{path}(\lambda_0) = \overline{\Cal}^{path}(\lambda_1).
$$
In the next section, we will elaborate this point further.

\section{Alexander isotopy of loops in $\CP^{ham}(Sympeo(D^2,\del D^2),id)$}
\label{sec:alexander}

For the description of Alexander isotopy, we need to consider the
conjugate action of rescaling maps of $D^2$
$$
R_a: D^2(1) \to D^2(a) \subset D^2(1)
$$
for $0 < a < 1$ on $Hameo(D^2,\partial D^2)$, where $D^2(a)$ is the disc
of radius $a$ with its center at the origin. We note that
$R_a$ is a conformally symplectic map and so its conjugate
action maps a symplectic map to a symplectic map
whenever it is defined.

Furthermore the right composition by $R_a$ defines a map
$$
\phi \mapsto \phi \circ R_a^{-1} \, ; \, Hameo(D^2(a),\partial D^2(a))  \subset Hameo(D^2,\partial D^2)
\to Homeo(D^2,\partial D^2)
$$
and then the left composition by $R_a$
followed by extension to the identity on $D^2 \setminus D^2(a)$ defines a map
$$
Hameo(D^2,\partial D^2) \to  Hameo(D^2,\partial D^2).
$$
We have the following important formula for the behavior of
Calabi invariants under the Alexander isotopy.

\begin{lem}\label{lem:Callambdaa} Let $\lambda \in \CP^{ham}(Sympeo(D^2,\del D^2),id)$ be a given
continuous Hamiltonian path on $D^2$. Suppose $\supp \lambda \subset D^2(1-\eta)$
for a sufficiently small $\eta > 0$. Consider the one-parameter family of maps $\lambda_a$
defined by
$$
\lambda_a(t,x) =
\begin{cases}a \lambda(t,\frac{x}{a})  \quad &\mbox{for $|x|
\leq a(1-\eta)$} \\
x \quad &\mbox{otherwise}
\end{cases}
$$
for $0 < a \leq 1$. Then $\lambda_a$ is also a topological
Hamiltonian path on $D^2$ and satisfies \be\label{eq:callamlama}
\overline{\Cal}^{path}(\lambda_a) = a^4
\overline{\Cal}^{path}(\lambda). \ee
\end{lem}
\begin{proof} A straightforward calculation proves that $\lambda_a$
is generated by the (unique) continuous Hamiltonian, which we denote by $\operatorname{Dev}(\lambda_a)$
following the notation of \cite{oh:hameo1,oh:hameo2}, which is defined by
\be\label{eq:Devlambdaa}
\operatorname{Dev}(\lambda_a)(t,x) = \begin{cases}
a^2 H \left(t, \frac{x}{a}\right)
\quad &\mbox{for $|x| \leq a(1-\eta)$} \\
0 \quad &\mbox{otherwise}
\end{cases}
\ee
where $H = \operatorname{Dev}(\lambda)$ : Obviously the right hand side function
is the hamiltonian-limit of $\operatorname{Dev}(\lambda_{i,a})$ for a sequence $\lambda_i$ of
smooth hamiltonian approximation of $\lambda$ where  $\lambda_{i,a}$ is defined
by the same formula for $\lambda_i$.

From these, we derive the formula
\beastar
\overline{\Cal}^{path}(\lambda_a)
& = & \lim_{i \to \infty} \Cal^{path}(\lambda_{i,a}) = \lim_{i \to \infty} a^4 \Cal^{path}(\lambda_i) \\
& = & a^4 \lim_{i \to \infty} \int_0^1 \int_{D^2}H_i(t,y)\Omega \wedge \, dt \\
& = & a^4 \lim_{i \to \infty} \Cal^{path}(\lambda_i)
=  a^4 \overline{\Cal}^{path}(\lambda).
\eeastar
This proves \eqref{eq:callamlama}.
\end{proof}

We would like to emphasize that the $s$-Hamiltonian $F_{\Lambda}$
of $\Lambda(s,t) = \lambda_s^t$ does not converge in $L^{(1,\infty)}$-topology and so we cannot
define its hamiltonian limit.

Explanation of this relationship is now in order
in the following remark.

\begin{rem}\label{rem:sHamiltonian}
Let $D^{2n} \subset \R^{2n}$ be the
unit ball. Consider a smooth Hamiltonian $H$ with $\supp \phi_H
\subset \operatorname{Int} D^{2n} \subset \R^{2n}$
and its Alexander isotopy
$$
\Lambda(s,t) = \phi_{H^s}^t = \lambda_s(t), \quad \lambda = \phi_H
$$
Denote by $H_\Lambda$ and $K_\Lambda$ the $t$-Hamiltonian and the
$s$-Hamiltonian respectively. Then we have Banyaga's formula $\frac{\del H}{\del s} = \frac{\del K}{\del t}
- \{H,K\}$ which is equivalent to
the formula
\be\label{eq:dtdK} \frac{\del K}{\del t} =
\frac{\del}{\del s}(H \circ \phi_{H^s}^t)\circ (\phi_{H^s}^t)^{-1}.
\ee
(See p.78 \cite{oh:jkms} for its derivation, for example.)
But we compute
$$
H_t\circ \phi_{H^s}^t(x) = s^2 H_t\left(\frac{\phi_{H^s}^t(x)}{s}\right)
= s^2 H\left(t,\frac{\phi_{H^s}^t(x)}{s}\right).
$$
Therefore we derive
\be\label{eq:FLambda}
K(s,t,x) = 2s \int_0^t
H\left(u,\frac{x}{s}\right) \, du + s \int_0^t \left\langle \left(d\overline
H\left(u,\frac{(\phi_{H^s}^u)^{-1}(x)}{s}\right)\right),
(\phi_{H^s}^u)^{-1}(x)\right\rangle\, du.
\ee
For the second summand, we use the identity $\overline H(t,x) = - H(t, \phi_H^t(x))$.
From this expression, we note that $K$ involves differentiating the
Hamiltonian $H_i$ and hence goes out of the $L^{(1,\infty)}$ hamiltonian category.
\end{rem}

Recall that the well-known Alexander isotopy on the disc $D^2$ exists in the homeomorphism category but not
in the differentiable category. We will establish that such an Alexander isotopy
defines contractions of topological Hamiltonian loops to the
identity constant loop in the topological Hamiltonian category.

\begin{thm}\label{thm:alexander} Let $\lambda$ be
a loop in $\CP^{ham}(Sympeo(D^2,\del D^2),id)$.
Define $\Lambda: [0,1]^2 \to Sympeo(D^2,\del D^2)$ by
$$
\Lambda(s,t) = \lambda_s(t).
$$
Then $\Lambda$ is a hameotopy between $\lambda$ and
the constant path $id$.
\end{thm}
\begin{proof} We have $\lambda \in \CP^{ham}(Sympeo(D^2,\del D^2),id)$
with $\lambda(0) = \lambda(1)$. Then $\lambda_s$ defines a loop
contained in $\CP^{ham}(Sympeo(D^2,\del D^2),id)$ for each $0 \leq s
\leq 1$.  Let $H_i$ be an approximating sequence of
the topological Hamiltonian loop $\lambda$.

We fix a sequence $\e_i \searrow 0$ and define a 2-parameter Hamiltonian family $\Lambda_{i,\e_i}$ defined by
\be\label{eq:Lambdaiei}
\Lambda_{i,\e_i}(s,t): = \lambda_{i,\chi_i(s)}(t, \cdot)\circ \lambda_{i,\e_i}^{-1}(t,\cdot)
\ee
where $\chi_i:[0,1] \to [\e_i,1]$ is a monotonically increasing surjective function
with $\chi_i(t) = \e_i$ near $t = 0$, $\chi_i(1) = 1$ near $t =1$, and $\chi_i \to id_{[0,1]}$ in the
hamiltonian norm (see Definition 3.19 and Lemma 3.20 \cite{oh:hameo1} for this fact).
It follows that the sequence $\Lambda_{i,\e_i}$ is smooth and uniformly converges in
hamiltonian topology as $i \to \infty$ over $s \in [0,1]$ and
$\Lambda_{i,\e_i} ^t(1) \to \lambda(t)$ since the Alexander
isotopy is smooth as long as $s > 0$ and by definition $\Lambda_{i,\e_i}$
involves the Alexander isotopy for $s \geq \e_i > 0$. The convergence immediately follows from the
explicit expression of $\lambda_a$ in Lemma \ref{lem:Callambdaa}.

Finally we need to check
\be\label{eq:DevLambdai}
\|\Dev(\Lambda_{i,\e_i}(s,\cdot)) - \Dev(\Lambda_{j,\e_j}(s,\cdot)\| \to 0
\ee
uniformly over $s \in [0,1]$ as $i, \, j \to \infty$. For this, we apply the standard formula
of $\Dev$ for the composed flow,
$$
\Dev(\lambda \mu^{-1})(t,x) = \Dev(\lambda)(t,x) - \Dev(\mu)(t,\mu_t \lambda_t^{-1}(x))
$$
to $\Lambda_{i,\e_i}: = \lambda_{i,\chi_i(s}(t, \cdot)\circ \lambda_{i,\e_i}^{-1}(t,\cdot)$, which
amounts to the more familiar formula $(H\#\overline G)_t = H_t - G_t\circ \phi_G^t(\phi_H^t)^{-1}$
in the literature. Then we get
\be\label{eq:Devlambdai}
\Dev(\Lambda_{i,\e_i}(s,\cdot))(t,x) = \Dev(\lambda_{i,\chi_i(s)})(t,x)
- \Dev(\lambda_{i,\e_i})(t, \lambda_{i,\e_i}^t\circ (\lambda_{i,\chi_i(s)}^t)^{-1}(x))
\ee
where
$$
\Dev(\lambda_{i,\chi_i(s)})(t,x)
= \begin{cases} \chi_i(s)^2 H_i (t, \frac{x}{\chi_i(s)}) \quad & \mbox{for } \, |x| \leq \chi_i(s) (1-\eta) \\
0 \quad &\mbox{otherwise}
\end{cases}
$$
and
$$
\Dev(\lambda_{i,\e_i})(t,x)
= \begin{cases} \e_i^2 H_i (t, \frac{x}{\e_i}) \quad & \mbox{for } \, |x| \leq \e_i (1-\eta) \\
0 \quad &\mbox{otherwise}.
\end{cases}
$$
From these expressions, \eqref{eq:DevLambdai} immediately follows. This finishes the proof.
\end{proof}

\begin{cor}\label{cor:alex} If $\lambda_0, \, \lambda_1 \in \CP^{ham}(Sympeo(D^2, \del D^2),id)$
and $\lambda_0(1) = \lambda_1(1)$, then they are hameotopic relative
to the end.
\end{cor}
\begin{proof} Theorem \ref{thm:alexander} implies that
the standard Alexander isotopy given in Lemma \ref{lem:Callambdaa}
is a hameotopy contracting
any topological Hamiltonian loop to the
identity in $\CP^{ham}(Sympeo(D^2,\del D^2),id)$ with ends points fixed.  This proves that the product loop
$\lambda_0\lambda_1^{-1}$, which is based at the identity, is contractible
via a hameotopy relative
to the ends. Then this implies that
$\lambda_0$ and $\lambda_1$ are hameotopic to each other relative to the ends.
\end{proof}

An immediate consequence of Corollary \ref{cor:alex}
is the following

\begin{prop} Suppose Conjecture \ref{conj:loop-vanishing} holds. Then
we have
$$
\overline{\Cal}^{path}(\lambda_0) = \overline{\Cal}^{path}(\lambda_1)
$$
if $\lambda_0, \, \lambda_1 \in \CP^{ham}(Sympeo(D^2, \del D^2),id)$
and $\lambda_0(1) = \lambda_1(1)$,.
\end{prop}

This theorem  implies that $\overline{\Cal}^{path}$
restricted to $\CP^{ham}(Sympeo(D^2,\del D^2),id)$
depends only on the final point and so gives rise to the following main theorem on the extension of
Calabi homomorphism.

\begin{thm} Suppose Conjecture \ref{conj:loop-vanishing} holds. Define a map
$
\overline{\operatorname{Cal}}: Hameo(D^2,\del D^2) \to \R
$
by
$$
\overline{\operatorname{Cal}}(g) := \overline{\Cal}^{path}(\lambda)
$$
for a (and so any) $\lambda \in \CP^{ham}(Sympeo(D^2,\del D^2),id)$ with
$g = \lambda(1)$.
Then this is well-defined and extends the Calabi homomorphism $\Cal:Diff^\Omega(D^2,\del D^2) \to \R$ to
$$
\overline{\operatorname{Cal}}: Hameo(D^2, \del D^2) \to \R.
$$
\end{thm}

Once this theorem is established, nonsimpleness of $Hameo(D^2, \del D^2)$
immediately follows from Conjecture \ref{conj:loop-vanishing}.
(See \cite{oh:hameo2} for the needed argument.)

\section{Reduction to the engulfed case and its Lagrangianization}
\label{sec:engulfed}

In this section, we reduce the proof of Conjecture \ref{conj:loop-vanishing} to the engulfed
topological Hamiltonian loops on $S^2$.
Using the given identification of $D^2$ as the upper hemi-sphere denoted by $D^2_+$, we can embed
$$
\iota^+: \CP^{ham}(Symp(D^2,\del D^2);id) \hookrightarrow
\CP^{ham}(Symp(S^2);id)
$$
by extending any element $\phi_H \in \CP^{ham}(Symp(D^2,\del D^2);id)$
to the one that is identity on the lower hemisphere  $D^2_-$ by setting $H \equiv 0$
thereon.

We first recall the definition of engulfed Hamiltonians from \cite{oh:local}.

\begin{defn} Let $(M,\omega)$ be a symplectic manifold.
Let a Darboux-Weinstein chart
$$
\Phi: \CV \subset T^*\Delta \to U_\Delta \subset (M\times M, \omega \oplus -\omega)
$$
be given. We call $\CU$ a Darboux-Weinstein neighborhood of the diagonal with respect to $\Phi$.
In general we call a neighborhood $U_\Delta$ of the diagonal a \emph{Darboux-Weinstein neighborhood}
if it is the image of a Darboux-Weinstein chart.
\end{defn}

With this preparation, we are ready to recall the following definition from \cite{oh:local}.

\begin{defn}\label{defn:engulfed}
\begin{enumerate}
\item An isotopy of Lagrangian submanifold $\{L_t\}_{0 \leq s \leq 1}$
of $L$ is called \emph{$V$-engulfed} if there exists a Darboux
neighborhood $V$ of $L$ such that $L_s \subset V$ for all $s$.
When we do not specify $V$, we just call the isotopy engulfed.
\item  We call a (topological) Hamiltonian path $\phi_H$ $\CU$-engulfed if its graph
$\Graph \phi_H^t$ is engulfed in a Darboux-Weinstein neighborhood $\CU$ of the
diagonal $\Delta$ of $(M \times M, \omega \oplus -\omega)$.
\end{enumerate}
\end{defn}

Now let $\lambda = \phi_F$ be a contractible topological
Hamiltonian loop contained in $\CP^{ham}(Sympeo(D^2,\del D^2),id)$
and $\Lambda = \{\lambda(s)\}_{s \in [0,1]}$ a given hameotopy contracting the loop.

Let $\lambda \in \CP^{ham}(Sympeo(D^2,\del D^2),id)$ and
consider its extension $\iota^+(\lambda)$ as an element in $\CP^{ham}(Sympeo_{D^+}(S^2),id)$
obtained via the embedding $\iota^+$. Denote by $D^1(T^*S^2)$ the unit cotangent bundle and
by $\overline \Delta$ the anti-diagonal
$$
\overline \Delta = \{(x,\overline x) \in S^2 \times S^2 \mid x \in S^2 \}.
$$
Then it is well-known that the geodesic flow of the standard metric on
$S^2$ induces a symplectic diffeomorphism
\be\label{eq:PhiD1}
\Phi: D^1(T^*S^2) \to S^2 \times S^2 \setminus \overline \Delta
\ee
where $\overline x$ is the involution along a (fixed) equator. We regard
the image $\CU = S^2 \times S^2 \setminus \overline \Delta$ as a Darboux-Weinstein
neighborhood of the diagonal $\Delta \subset S^2 \times S^2$.

It is then easy to see the following

\begin{lem}
Let $\lambda \in \CP^{ham}(Sympeo(D^2,\del D^2),id)$ and
denote by $\lambda^+ = \iota^+(\lambda) \in \CP^{ham}(Sympeo_{D^+}(S^2),id)$
constructed as above. Then
$$
(\lambda^+_t \times id) (\Delta) \cap \overline \Delta = \emptyset.
$$
In particular, the path $\lambda^+$ is $\CU$-engulfed.
\end{lem}

Motivated by the above discussion,
we will always consider only the engulfed case in the rest of the paper,
unless otherwise said.

Now let $F: [0,1] \times M \to \R$ be a mean normalized engulfed Hamiltonian
on a closed symplectic manifold $(M,\omega)$. The manifold $M$ carries a natural Liouville measure
induced by $\omega^n$. Consider the diagonal Lagrangian $\Delta \subset (M \times M,
\omega \oplus -\omega)$ identified with the zero section
$o_\Delta \subset T^*\Delta$ in a Darboux chart $(V_\Delta, -d\Theta)$ of
$\Delta$ in $M \times M$. Put a density $\rho_\Delta$ on $\Delta \subset M \times M$ induced
by $\omega^n$ by the diffeomorphism of the second projection $\pi_2:\Delta \to M$.

We fix Darboux neighborhoods
$$
V_\Delta \subset \overline V_\Delta \subset U_\Delta
$$
and let $\omega\oplus -\omega = -d\Theta$ on $U_\Delta$ regarded as
a neighborhood of the zero section of $T^*\Delta$ once and for all.
Then
$$
\Graph \phi_F^t \subset V_\Delta \quad \mbox{for all $t \in [0,1]$}.
$$
Here we define
$$
\Graph \phi_F^t : = \{(\phi_F^t(y),y) \mid y \in M\}.
$$

We consider the Hamiltonian $\pi_1^*F$, i.e., the one defined by
$$
\pi_1^*F(t,(x,y)) = F(t,x)
$$
on $T^*\Delta$. This itself is not supported in $U_\Delta$ but we can multiply a cut-off
function $\chi$ of $U_\Delta$ so that
$$
\chi \equiv 1 \quad \mbox{on } \, V_\Delta, \quad \supp \chi \subset U_\Delta
$$
and consider the function $\mathbb F$ defined by
$$
\mathbb F(t,(x,y)) = \chi(x,y) \pi_1^*F(t,(x,y))= \chi(x,y) F(t,x)
$$
so that
the associated Hamiltonian deformations of $\psi^t(o_N)$ are unchanged.
We note that \emph{$\mathbb F$ is compactly supported in $T^*\Delta$.}
and automatically satisfies the normalization condition
\be\label{eq:normalized}
\int_\Delta \mathbb F(t, \phi_\mathbb F^t(q))\, \rho_\Delta = 0
\ee
for all $t \in [0,1]$ where $\rho_\Delta$ is the measure on $\Delta$ induced by the Liouville measure
on $M$ under the projection $\pi_2:\Delta \subset M \times M \to M$.

Now we denote by $f_\mathbb F$ the basic phase function of $\Graph \phi_F^1 = \phi_{\mathbb F}^1(o_\Delta)$.
In the next section, we will examine the relationship between this function
and the Calabi invariant of $F$.

\section{Basic phase function $f_H$ and its axioms}
\label{sec:basicphase}

In this section, we first recall the definition of \emph{basic phase
function} constructed in \cite{oh:jdg} and summarize its axiomatic properties.
Following the terminology of \cite{PPS}, we first introduce the following definition.

\begin{defn}
Let $L \subset T^*N$ be a Hamiltonian deformation of the zero section $o_N$. We call
any continuous function $f: N \to \R$ a \emph{graph selector} such that
$$
(q,df(q)) \in L
$$
where $df(q)$ exists.
\end{defn}

Existence of such a single-valued continuous function was proved by
Sikorav, Chaperon \cite{chaperon} by the generating function method
and by the author \cite{oh:jdg} using the Lagrangian Floer theory.
Lipschitz continuity of this particular graph selector follows from
the continuity result established in section 6 \cite{oh:jdg} specialized to
the submanifold $S$ to be a point. The detail of another proof of this Lipschitz continuity
is also given in \cite{PPS} using the generating function
techniques.

We denote by $\operatorname{Sing} f$
the set of non-differentiable points of $f$. Then by definition
$$
N_0=\operatorname{Reg} f: = N \setminus \operatorname{Sing} f
$$
is a subset of full measure and $f$ is differentiable thereon. In fact, for a generic
choice of $L = \phi_H^1(o_N)$, $N_0$ is open and dense and $\operatorname{Sing} f$
is a stratified submanifold of $N$ of codimension at lease 1. (See \cite{oh:lag-spectral} for
its proof.)

By definition,
\be\label{eq:dfL} |df(q)| \leq \max_{x \in L} |p(x)|
\ee
for any $q \in N_0$, where $x = (q(x),p(x))$ and the norm
$|p(x)|$ is measured by any given Riemannian metric on $N$.

The following is an immediate corollary of the definition.
We denote by $d_{\text{\rm H}}$ the Hausdorff distance.

\begin{cor}\label{cor:|df|to0}
As $d_{\text{\rm H}}(\phi_H^1(o_N),o_N) \to 0$,
$|df(q)| \to 0$ uniformly over $q \in N_0$.
\end{cor}

However this result itself does not tell us much about the convergence of
the values of the function $f$ itself because a priori the value of $f$ might not be
bounded for a sequence $H_i$ such that  $d_{\text{\rm H}}(\phi_H^1(o_N),o_N) \to 0$.

In \cite{oh:jdg}, a canonical choice of $f$ is constructed via the
chain level Floer theory, \emph{provided} the generating Hamiltonian
$H$ of $L=\phi_H^1(o_N)$ is given. The author called the corresponding graph selector $f$
the \emph{basic phase function} of $L = \phi_H^1(o_N)$ and denoted it by $f_H$.
We give a quick outline of the construction referring the
readers to \cite{oh:jdg} for the full details of the construction.

Consider the Lagrangian pair
$$
(o_N, T^*_qN), \quad q \in N
$$
and its associated Floer complex $CF(H;o_N, T^*_qN)$ generated by
the Hamiltonian trajectory $z:[0,1] \to T^*N$ satisfying
\be\label{eq:Hamchordeq2}
\dot z = X_H(t,z(t)), \quad z(0) \in o_N, \, z(1) \in T^*_qN.
\ee
Denote by $\CC hord(H;o_N,T^*_qN)$ the set of
solutions. The differential $\del_{(H,J)}$ on $CF(H;o_N, T^*_qN)$ is
provided by the moduli space of solutions of the perturbed
Cauchy-Riemann equation
\be\label{eq:CRHJqN}
\begin{cases}
\dudtau + J\left(\dudt - X_H(u) \right) = 0 \\
u(\tau,0) \in o_N, \, u(\tau,1) \in T^*_qN.
\end{cases}
\ee

An element $\alpha \in CF(H;o_N,T^*_qN)$ is expressed as a finite
sum
$$
\alpha = \sum_{z \in \CC hord(H;o_N,T_q^*N)} a_z [z], \quad a_z \in
\Z.
$$
We denote the level of the chain $\alpha$ by
$$
\lambda_H(\alpha): = \max_{z \in \supp \alpha} \{\CA^{cl}_H(z)\}.
$$
The resulting invariant $\rho^{lag}(H;[q])$ is to be defined by the mini-max
value
$$
\rho^{lag}(H;[q]) = \inf_{\alpha \in [q]}\lambda_H(\alpha)
$$
where $[q]$ is a generator of the homology group
$HF(o_N, T^*_qN) \cong \Z$.

 A priori, $\rho^{lag}(H;[q])$ is defined when
$\phi_H^1(o_N)$ intersects $T_q^*N$ transversely but can be extended
to non-transversal $q$'s by continuity. By varying $q \in N$, this
defines a function $f_H: N \to \R$ which is precisely the one called the basic
phase function in \cite{oh:jdg}.

\begin{prop}[Section 7 \cite{oh:jdg}]
There exists a solution $z:[0,1] \to T^*N$ of $\dot z = X(t,z)$ such
that $z(0)=q, \, z(1) \in o_N$ and $\CA^{cl}_{H}(z) = \rho^{lag}(H;\{q\})$
\emph{whether or not $\phi_H^1(o_N)$ intersects $T_q^*N$
transversely.}
\end{prop}

We summarize the main properties of $f_H$ established in
\cite{oh:jdg}.

\begin{prop}[Theorem 9.1 \cite{oh:jdg}]
When the Hamiltonian $H=H(t,x)$ such that $L = \phi_H^1(o_N)$ is
given, there is a canonical lift $f_H$ defined by $f_H(q): =
\rho^{lag}(H;\{pt\})$ that satisfies \be\label{eq:fversush} f_H \circ
\pi(x) = h_H(x) = \CA^{cl}_H(z_x^H) \ee for some Hamiltonian chord
$z_x^H$ ending at $x \in T^*_qN$. This $f_H$ satisfies the following
property in addition
\be\label{eq:fH}
\|f_H - f_{H'}\|_\infty \leq \|H - H'\|.
\ee
\end{prop}

An immediate corollary of this proposition is the following proved in \cite{oh:jdg,oh:homotopy}.

\begin{cor}\label{cor:fHiconv} If $H_i$ converges in $L^{(1,\infty)}$, then $f_{H_i}$ converges
uniformly.
\end{cor}

\begin{rem} We would like to emphasize that there is no such $C^0$-control of the
basic generating function $h_H$ even when $H \to 0$ in hamiltonian topology.
\end{rem}

Based on the above proposition, we define

\begin{defn}\label{defn:fH} Denote by $H^a$ the Hamiltonian generating the
rescaled isotopy $t \mapsto \phi_H^{at}$ for $a > 0$.
For any given topological Hamiltonian $H=H(t,x)$, we define its
timewise basic phase function by
\be\label{eq:limitfH}
{\bf f}_H(t,x): = \lim_{i \to \infty}f_{H_i^t}(x)
\ee
for any approximation sequence $H_i$ of $H$.
\end{defn}
We will always denote a parametric version in bold-faced letters.

We note that the basic generating function $h_{H_i}$ could behave wildly as a
whole. But Proposition \label{prop:fHiconv} restricted to the basic Lagrangian
selector converges nicely.
Note that $\pi_H= \pi|_{L_H}: L_H = \phi_H^1(o_N) \to N$ is surjective for all $H$ and so
$\pi_H^{-1}(q) \subset o_N$ is a non-empty
compact subset of $o_N \cong N$. Therefore we can regard the
`inverse' $\pi_H^{-1}:N \to L_H \subset T^*N$
as an everywhere defined multivalued section of $\pi: T^*N \to N$.

We introduce the following general definition

\begin{defn} Let $L \subset T^*N$ be a Lagrangian submanifold projecting
surjectively to $N$. We call a single-valued section $\sigma$ of $T^*N$ with values
lying in $L$ a \emph{Lagrangian selector} of $L$.
\end{defn}

Once the graph selector $f_H$ of $L_H$ is picked out, it provides a
natural Lagrangian selector defined by
$$
\sigma_H(q): = \text{Choice}\{x \in L_H \mid \pi(x) = q, \,  \CA^{cl}_H(z^H_x) = f_H(q)\}
$$
via the axiom of choice where $\text{Choice}$ is a choice function. It satisfies
\be\label{eq:sigmaHdfH}
\sigma_H(q) = df_H(q)
\ee
whenever $df_H(q)$ is defined. We call this particular Lagrangian selector of
$L_H$ the \emph{basic Lagrangian selector.}
The general structure theorem of the wave front (see
\cite{eliash:front}, \cite{PPS} for example) proves that the section
$\sigma_H$ is a differentiable map on a set of full measure for a
generic choice of $H$ which is, however, \emph{not necessarily
continuous}: This is because as long as $q \in N \setminus
\operatorname{Sing}f_H$, we can choose a small open neighborhood of
$U \subset N \setminus \operatorname{Sing}f_H$ of $q$ and $V \subset
L_H = \phi_H^1(o_N)$ of $x \in V$ with $\pi(x) = q$ so that the
projection $\pi|_V: V \to U$ is a diffeomorphism.

\section{Calabi homomorphism and basic phase function}
\label{sec:extension}

Suppose $F$ is a topological Hamiltonian and $F_i$ its approximating sequence
and define $\mathbb F_i$ and $\mathbb F$ as in section \ref{sec:engulfed}.

We first prove the following general theorem in arbitrary dimension. We recall
that  $f_{\underline{\mathbb F_i}}$ converges to $f_{\underline{\mathbb F}}$ uniformly.

\begin{thm}\label{thm:average=Cal} Let $\lambda = \phi_F$ be any
contractible topological Hamiltonian loop
in $\CP^{ham}(Sympeo_U(M,\omega),id)$  and
with $U = M \setminus B$  where $B$ is a closed subset of nonempty interior.
Choose an approximating sequence $F_i$. Denote by
$$
\overline \Cal(F) = \int_0^1 \int_M F\, \mu_\omega \, dt
$$
for the Liouville measure associated to $\omega$. Then
\be\label{eq:limitfH}
f_{\underline{\mathbb F}}(x)\equiv \frac{\overline \Cal(F)}{\vol_\omega(M)}
\ee
for all $x \in M$.
\end{thm}
\begin{proof}
Let $\underline{F_i}= \Dev(\phi_{F_i})$ which is given by
$$
\underline{F_i}(t,x) = F_i(t,x) - c_i(t)
$$
where
$$
c_i(t) = \frac{1}{\vol_\omega(M)} \int_M F_i(t,x)\, \mu_\omega.
$$
Then we have
\be\label{eq:underlineFi}
\underline{F_i}(t,x) \equiv - c_i(t)
\ee
and so
$$
\int_0^1 \underline{F_i}(t,x)\, dt = - \int_0^1 c_i(t) \, dt = - \frac{\Cal_U(F_i)}{\vol_\omega(M)}
$$
for all $x \in B$.

Since $F_i$ is an approximating sequence of topological Hamiltonian $F$, it follows
$F_i \to F$ in $L^{(1,\infty)}$-topology.
Therefore applying \eqref{eq:fH} to $H=F_i$ and $H' = 0$ and using
the convergence $\|F - F_i\| \to 0$ as $i \to \infty$, we obtain the inequality
$$
|f_{\mathbb F_i}| \leq \|F\| + \frac{1}{2}
$$
for all sufficiently large $i$'s.

\emph{Here now enters in a crucial way the fact that $\phi_F$ generates a topological Hamiltonian
loop, not just a path.}
Together with the Lipschitz property of $f_{\mathbb F_i}$ and the inequality (see \eqref{eq:dfL})
$$
|df_{\underline{\mathbb F_i}}| \leq \overline d(\phi_{\underline{F_i}}^1,id) \to 0,
$$
it immediately follows from the co-area formula (see Theorem 1 of section 3.4.2 \cite{evans},
for example)
that we can choose a subsequence, again denoted by $F_i$, so that
$f_{\underline{\mathbb F_i}} \to c$ uniformly for some constant $c$.

Therefore it remains to show that this constant is indeed the value $\frac{\overline \Cal(F)}{\vol_\omega(M)}$.
Denote $K = \supp F$ which is a compact subset of $U = M \setminus B$.
We now recall the definition of Hamiltonian topology on noncompact manifolds,
Definition \ref{defn:directlimit}. By definition, there exists $\delta > 0$ such that
$$
\supp F_i \subset \Int K(1+\delta/2) \subset K(1+\delta) \subset U
$$
where $K(1+\delta)$ is the (closed)
$\delta$-neighborhood of $K$ for all sufficiently large $i$'s. In particular,
\be\label{eq:Bdelta/2}
B(1+\delta/2) \subset M \setminus K(1+\delta/2).
\ee

For any such $i$'s, we also have
$$
F_i \equiv  0, \quad \phi_{F_i}^t \equiv id
$$
on $B(1+\frac{\delta}{2})$. In particular,
$$
\Graph F_i \cap o_\Delta \supset o_{B(1+\frac{\delta}{2})}.
$$
Therefore the same properties stated above as for $F_i$ still hold for $\underline{F_i}$
except the values thereof on $B$ are changed to $ -c_i(t)$.

Let ${\bf q} \in o_B$ be any point in its interior.
By the spectrality of the values of $f_{\underline{{\mathbb F}_i}}({\bf q})$ (Theorem 5.3 \cite{oh:lag-spectral}),
there is a point ${\bf x} \in T_{\bf q}^*M \cap \Graph \phi_{F_i}^1$ such that
$(\phi_{\mathbb F_i}^1)^{-1}({\bf x}) \in o_\Delta$
and
$$
f_{\underline{\mathbb F_i}}({\bf q}) = \CA^{cl}(z_{\bf x}^{\underline{\mathbb F_i}}).
$$
We denote $(\phi_{\mathbb F_i}^1)^{-1}({\bf x})  = (q', q')$.

 Because of this, $\phi_{F_i}^1 \to id$ as $i \to \infty$ by definition of
the approximating sequence $F_i$ of $F$.
Combining this with ${\bf q} = (q,q) \in \Int o_B$, $\pi_\Delta({\bf x}) = {\bf q}$, we derive
$$
d((\phi_{\mathbb F_i}^1)^{-1}({\bf x}), {\bf x}), \, d({\bf x}, \pi_\Delta({\bf x})) < \frac{\delta}{4}
$$
for all sufficiently large $i$'s. Then $d((\phi_{\mathbb F_i}^1)^{-1}({\bf x}), (q,q)) < \frac{\delta}{2}$.
Since $\underline{\mathbb F_i}(t,{\bf x})
= \underline{F_i}(t,x)$ for ${\bf x} = (x,y)$,
the associated Hamiltonian trajectory $z_{\bf x}^{\underline{\mathbb F}}$ has the form $(\phi_{F_i}^t(q'),q')$
where $(\phi_{\mathbb F_i}^1)^{-1}({\bf x}) = (q',q')$. But $d(q,q') < \frac{\delta}{2}$ and hence
$q' \in M \setminus B(1+\delta) \subset K(1+\frac{\delta}{2})$. (We refer to the proof of Lemma 7.5 \cite{oh:lag-spectral} for
a similar argument used for a similar purpose.)

Therefore $\phi_{F_i}^t(q') \equiv q'$ for all
$t \in [0,1]$. This proves that $z_{\bf x}^{\mathbb F}$ must be the constant
trajectory $z_{\bf x}^{\mathbb F}(t) \equiv {\bf q}$. Then we compute its action
value
\beastar
f_{\underline{\mathbb F_i}}({\bf q}) & = &
\CA^{cl}(z_{\bf x}^{\underline{\mathbb F_i}}) \\
& = & - \int_0^1 \underline{\mathbb F_i}(t, {\bf q})\, dt
= - \int_0^1 \underline{F_i}(t, q)\, dt = \int_0^1 c_i(t)\, dt  = \frac{\Cal_U(F_i)}{\vol_\omega(M)}.
\eeastar
Since $F_i \to F$ in $L^{(1,\infty)}$-topology and $\supp \phi_{F_i}, \, \supp \phi_F \subset U$,
it also follows $\Cal_U(F_i) \to \overline{\Cal_U}(F)$ as $i \to \infty$. This proves
indeed $f_{\underline{\mathbb F_i}} \to \frac{\overline{\Cal}_U(F)}{\vol_\omega(M)}$.
\end{proof}

 An examination of the argument at the end of the proof leading to the
identification of the constant with the $\frac{\overline{\Cal}_U(F)}{\vol_\omega(M)}$ shows that the reason why
the convergence $\phi_{F_i}^1 \to id$ enters is because we need for the projection
$\pi_\Delta({\bf x})$ to lie outside $\supp F_i$ to get the required identification. This needed property automatically holds
for the projection $U_\Delta \to \Delta$ with $B = D^2_-$ of the
canonical Darboux-Weinstein neighborhood obtained through the embedding \eqref{eq:PhiD1} in section \ref{sec:engulfed}.
This is because under this embedding the projection $\pi_\Delta(x,y)$ is nothing but the mid-point projection
of $(x,y)$ along the geodesic connecting the points $x, \, y \in S^2$. Since
the upper hemisphere $D^2_+ \subset S^2$ is gedesically convex, $\pi_\Delta(x,y)$ is
always contained in $\Int D^2_+$ whenever $x, \, y \in \Int D^2_+$.
In particular $\pi_\Delta(\phi_{F_i}^1(q'),q') \in \Int D^2_+$
if $q' \in \Int D^2_+$ and hence the point $\pi_\Delta({\bf x}) = \pi_\Delta(\phi_{F_i}^1(q'),q')$ cannot be
projected to a point $(q,q)$ with $q \in B = D^2_-$ \emph{irrespective of the convergence $\phi_{F_i}^1 \to id$}.
This eliminates the above somewhat subtle argument for the case of our main interest.
An implication of this consideration leads to the following stronger result for this
case in that it applies to an arbitrary \emph{path} not just to loops.

\begin{thm}\label{thm:disccase}
Let $\lambda = \phi_F$ be any topological Hamiltonian path supported in
$\Int D^2$. Denote by $\mathbb F$ the associated Hamiltonian on $D^1(T^*\Delta_{S^2}) \cong S^2 \times S^2 \setminus \overline
\Delta_{S^2}$ constructed as before (via the embedding \eqref{eq:PhiD1}). Then
$$
f_{\underline{\mathbb F}}(x) = \frac{\overline\Cal(F)}{\vol(S^2)}
$$
for all $x \in D^2_-$.
\end{thm}

Of course, in this case, $f_{\underline{\mathbb F}}$ will not be constant on $D^2_+$
in general.

\section{Extension of Calabi homomorphism}
\label{sec:vanishing}

We recall from the definition of $\CP^{ham}(Symp_U(M,\omega),id)$
with $U = M \setminus B$ that if $\phi_F \in \CP^{ham}(Symp_U(M,\omega),id)$, then
there exists a 2-parameter Hamiltonian $H = H(s,t,x)$ such that
$\phi_{H(s)}^t \equiv id$ and $H \equiv 0$
on $B = M \setminus U$ for a nonempty open subset of $M$. In particular,
we have $\underline{H}(s) \equiv c(s)$ on $B$ with
$$
c(s) = \frac{1}{\vol_\omega(M)}\int_M H(s)\, \omega.
$$
Engulfedness of $H$ enables us to do computations on a Darboux-Weinstein neighborhood $V_\Delta$ of the diagonal
$\Delta \subset M \times M$, which we regard either as a subset of
$M\times M$ or that of $T^*\Delta$ depending on the given circumstances.
At the end, we will apply the computations to the given approximating sequence of
hameotopy of contractible topological Hamiltonian loop.

Now we further specialize to the case of our main interest $D^2$. We embed $D^2$
into $S^2$ as the upper hemisphere $D^2_+$ and denote $B = D^2_-$, the lower hemisphere.

The following is the main conjecture to beat which was originally proposed in \cite{oh:homotopy}.
This is the only place where the restriction to the two-disc $D^2$ is needed, but we expect
the same vanishing result hold for higher dimensional disc $D^{2n}$ or even for general
pair $(M,B)$, which is a subject of future study.

\begin{conj}\label{conj:signchange} Assume $M = S^2$ and
$B = D^2_-$ be the lower hemisphere as above.
Let $\Lambda= \left\{\phi_{H(s)}^t\right\}_{(s,t) \in [0,1]^2}$ be a hameotopy
contracting a topological Hamiltonian loop
$\phi_F$ with $F = H(1)$ such that $\phi_{H(s)}^1 \equiv \phi_{H(0)}^t \equiv id$
for all $t, \, s \in [0,1]$.

Let $f_{\underline{\mathbb F}}$ be the limit basic phase function defined by $f_{\underline{\mathbb F}}
= \lim_{i \to \infty} f_{\underline{\mathbb F_i}}$. Then
$$
f_{\underline{\mathbb F}} = 0.
$$
\end{conj}

Combining Theorem \ref{thm:average=Cal} and Conjecture \ref{conj:signchange},
we now prove the following.

\begin{thm}\label{thm:descent} Suppose Conjecture \ref{conj:signchange} holds. Then
the homomorphism $\overline{\Cal}^{path}: \CP^{ham}(Sympeo(D^2,\del D^2),id) \to \R$
descends to a homomorphism
$$
\overline{\Cal}: Hameo(D^2,\del D^2) \to \R
$$
which restricts to $\Cal: Ham(D^2,\del D^2) \to \R$.
\end{thm}
\begin{proof} Let $\phi \in Hameo(D^2,\del D^2)$. We will show that
for any topological Hamiltonian paths $\lambda, \, \lambda'$ with
$\lambda(1) = \lambda'(1) = \phi$,
$\overline{\Cal}^{path}(\lambda) = \overline{\Cal}^{path}(\lambda')$
By the homomorphism property, it is enough to prove
$\overline{\Cal}^{path}(\lambda^{-1}\lambda') = 0$. But we have
$\lambda^{-1}(0)\lambda'(0) = \lambda^{-1}(1)\lambda'(1)) = id$, i.e,
the path $\lambda^{-1}\lambda'$ defines a topological Hamiltonian loop based at the
identity. Therefore Conjecture \ref{conj:signchange} and Theorem \ref{thm:disccase} imply
$\overline{\Cal}^{path}(\lambda^{-1}\lambda') = 0$.

Then the theorem follows by defining $\overline{\Cal}: Hameo(D^2,\del D^2) \to \R$ to be
$$
\overline{\Cal}(\phi) = \overline{\Cal}^{path}(\lambda)
$$
for a (and so any) topological Hamiltonian path $\lambda$ with $\lambda(1) = \phi$.
\end{proof}

Therefore we have proved

\emph{So the main remaining task is to prove Conjecture \ref{conj:signchange} which will
prove all the conjectures stated in the present paper.}
In the next section, we will prove the conjecture for the \emph{weakly graphical} topological Hamiltonian loop
on the disc.

\part{Weakly graphical topological Hamiltonian loops on $D^2$}

\section{Geometry of graphical symplectic diffeomorphisms in 2-dimension}

We start with the following definition in general dimension.

\begin{defn} Let $\Psi: U_\Delta \to \CV$ be a Darboux-Weinstein chart of the diagonal $\Delta \subset M \times M$
and $\pi_\Delta: U_\Delta \to \Delta$ the associated projection.
\begin{enumerate}
\item We call an engulfed symplectic diffeomorphism $\phi: M \to M$ \emph{$\Psi$-graphical} if the projection
$\pi_\Delta$ is one-one, and an engulfed symplectic isotopy $\{\phi^t\}$ \emph{$\Psi$-graphical} if
each element $\phi^t$ $\Psi$-graphical. We call a Hamiltonian $F=F(t,x)$ $\Psi$-graphical if
its associated Hamiltonian isotopy $\phi_F^t$ $\Psi$-graphical.
\item We call a topological Hamiltonian loop $F$ is strongly (resp. weakly) $\Psi$-graphical, if it allows an
approximating sequence $F_i$ each element of which is $\Psi$-graphical (resp. whose time-one map $\phi_{F_i}^1$
is $\Psi$-graphical).
\end{enumerate}
\end{defn}
Denote by $F^a$ the time-dependent Hamiltonian generating the
path $t \mapsto \phi_F^{at}$. The statement (2) of this definition is
equivalent to saying that each $F^a$ is $\Psi$-graphical for $a \in [0,1]$.
We remark that any symplectic diffeomorphisms sufficiently $C^1$-close to the identity
is graphical, but not every $C^0$-close one.

In the rest of this section, we restrict ourselves to the two dimensional case.

We identify $U_y \times U_y \hookrightarrow T^*\Delta$ by the explicit linear coordinate changes
\be\label{eq:qptox}
{\bf q}_1 = \frac{q+Q}{2}, \, {\bf q}_2 = \frac{p+P}{2}, \, {\bf p}_2 = q-Q,\, {\bf p}_1 = P-p
\ee
where $(Q,P) = (Q,P) \circ \pi_1$ and $(q,p) = (Q,P) \circ \pi_2$ in this Darboux-Weinstein chart.
(We note that this chart can be chosen globally on $D^2$.)
Then we have
\bea\label{eq:QPqp}
Q & = & {\bf q}_1 - \frac{{\bf p}_2}{2}, \quad q = {\bf q}_1 + \frac{{\bf p}_2}{2} \nonumber\\
P & = & {\bf q}_2 + \frac{{\bf p}_1}{2}, \quad p = {\bf q}_2 - \frac{{\bf p}_1}{2}.
\eea
In short, we write
$$
x=(Q,P) = {\bf q} + \frac{1}{2} j{\bf p}, \quad y=(q,p) = {\bf q} - \frac{1}{2}j{\bf p}
$$
where $j: \R^2_{\bf p} \times \R^2_{\bf p}$ is the linear map given by
$j({\bf p}_1,{\bf p}_2) = (-{\bf p}_2, {\bf p}_1)$.

In dimension 2, we prove the following interesting phenomenon. Although we have not checked it,
it is unlikely that similar phenomenon occurs in higher dimensions. This theorem
has its own interest. The theorem itself will not be used in the proofs of main results of the present paper
except that the same kind of the proof will be used later in the proof of Proposition
\ref{prop:onetoonekappa}.

\begin{thm}\label{thm:onetoone} Suppose $\phi: M \to M$ is a $\Psi$-graphical symplectic diffeomorphism
and let $\Graph \phi = \Image \alpha_\phi$ for a closed one-form
$\alpha_\phi$ on $\Delta$. Then for any $0 \leq r \leq 1$,
the projection $\pi_2: M \times M \to M$ restricts to a one-one map
to $ \Image r\, \alpha_\phi \subset M \times M$. In particular
\be\label{eq:tdfphit}
\Image r\, \alpha_\phi = \Graph \phi_r
\ee
for some symplectic diffeomorphism $\phi_r: M \to M$ for each $ 0 \leq r \leq 1$.
\end{thm}
\begin{proof} We have only to prove the map
\be\label{eq:composition}
{\bf q} \mapsto {\bf q} - \frac{r}{2} j \alpha_\phi({\bf q})
\ee
is one-one. This is because it is the composition of the maps
$$
\Delta \to \Image \alpha_\phi; \quad {\bf q} \mapsto ({\bf q}, r\, \alpha_\phi({\bf q}))
$$
and the projection $\pi_2: \Image r\, \alpha_\phi \to M$ where the first
map is a bijective map. Denote this map by $\psi_r$.

Since the map $\psi_r$ has degree 1, it will be enough to prove
that it is an immersion since the latter will imply that the map must be a
covering projection. Therefore we need to prove that the derivative
$$
d\psi({\bf q}) = I - \frac{r}{2} j\, \nabla\alpha_\phi({\bf q})
$$
is invertible for all ${\bf q}$ and $ 0 \leq r \leq 1$. Here $\nabla\alpha_\phi$
is the covariant derivative of the one-form $\alpha_\phi$ with respect to
the flat affine connection $\nabla$.
We regard it as a section of $Hom(T\Delta, T^*\Delta)$, i.e., a bundle map
$$
\nabla\alpha_\phi: T\Delta \to T^*\Delta.
$$
\begin{lem} At each point ${\bf q} \in \Delta$, the linear map
$$
\nabla : v \mapsto \nabla_v \alpha_\phi
$$
is a symmetric operator, i.e., it satisfies
\be\label{eq:symmetry}
\langle \nabla_v \alpha_\phi, w \rangle = \langle \nabla_w \alpha_\phi, v \rangle
\ee
for all $v, \, w  \in T_{\bf q}\Delta$ at any ${\bf q} \in \Delta$.
\end{lem}
\begin{proof} This immediately follows from the fact that any
closed one-form can be locally written as $\alpha_\alpha = df_\phi$
for some function on $\Delta$. Then $\nabla\alpha_\phi = D^2 f_\phi$
which is the Hessian of the function $f_\phi$ which is obviously symmetric.
\end{proof}

We first prove the following general result on the set the set of $2\times 2$ symplectic matrices.

\begin{lem}\label{lem:starshape} Let $A$ be a $2 \times 2$ symmetric matrix.
Then
\be\label{eq:id+A}
\det(I - r j \, A) > 0
\ee
for all $r \in [0,1]$, provided it holds at $r =1$, i.e., provided
$$
\det(I - j \, A) > 0.
$$
The same holds for the opposite inequality.
\end{lem}
\begin{proof} Denote $A = \left(\begin{matrix} a & c \\
c & b \end{matrix}\right)$. Then straightforward computation shows
$$
jA = \left(\begin{matrix} 0 & -1 \\
1 & 0 \end{matrix}\right) \left(\begin{matrix} a & c \\
c & b \end{matrix}\right) = \left(\begin{matrix} -c & -b \\
a & c \end{matrix}\right).
$$
In particular $\tr(j\, A) = 0$ and hence
$$
\det(I - j \, A) = 1 + \det (j\, A) = 1 + (ab - c^2).
$$
Therefore $\det(I + j \, A) > 0$ is equivalent to
$$
1 + (ab - c^2) > 0.
$$
For $r = 0$, $I + r j \, A = I$ and so the inequality obviously holds.
On the other hand, if $r \in (0, 1]$, we derive
$$
1 + r^2 (ab - c^2) \geq r^2(1 + (ab - c^2)) > 0
$$
which finishes the proof.
\end{proof}

\begin{rem} Note that if $A$ is symmetric, then $jA \in sp(2)$ the Lie algebra of
the symplectic group $Sp(2)$. Then the set  $\{B \in sp(2) \mid \det (I - B) = 0\}$
is given by the equation
$$
1 + (ab - c^2) = 0; \quad B = \left(\begin{matrix} c & b \\
-a & - c \end{matrix}\right)
$$
which defines a hypersurface in $sp(2)$. If we denote
$$
sp_\pm(2) = \{B \in sp(2) \mid \pm \det(I - B) > 0\}
$$
what this lemma shows that each component thereof is star-shaped centered at $I$.
\end{rem}

By the hypothesis, it follows that $\psi=\psi_1$ is an orientation preserving diffeomorphism and so
$\det d\psi({\bf q}) > 0$. We now compute
$$
\det d\psi({\bf q}) = \det \left(I - \frac{1}{2} \nabla \alpha_\phi({\bf q})\right)
$$
and
$$
d\psi_r({\bf q}) =  I - \frac{r}{2} \nabla \alpha_\phi({\bf q}).
$$
By Lemma \ref{lem:starshape}, we derive $d\psi_r({\bf q}) > 0$ and so $\psi_r$ is immersed
for all $r$. This finishes the proof of Theorem \ref{thm:onetoone}.
\end{proof}

\begin{rem}
A similar proof also gives rise to the following theorem with $v$ and $r$
replaced, whose proof will be omitted since it will not be used in the
present paper: Suppose $\phi: M \to M$ is a $\Psi$-graphical symplectic diffeomorphism.
Consider the family of maps $\phi_v: M \to M$ defined by
$\phi_v(y) = y + v(\phi(y) - y)$ for $v \in [0,1]$.
Then $\phi_v$ is also $\Psi$-graphical for all $v$, i.e., we can express
$$
\Graph \phi_v = \Image \alpha_v
$$
for some one-form $\alpha_v$ on $\Delta$ for all $v$.
\end{rem}

\section{Weakly graphical Hamiltonian diffeomorphism and Alexander isotopy}
\label{sec:weak-alexander}

The following proposition reflects some special characteristic of Alexander isotopy
relative to the general hameotopy.

\begin{prop}\label{prop:onetoonekappa} Suppose that $\phi_{F_i}^1$ is $\Psi$-graphical. Then $\phi_{F_i,a}^1$
defined as in Lemma \ref{lem:Callambdaa} is also $\Psi$-graphical for all $0 \leq a \leq 1$.
\end{prop}
\begin{proof} The proof of this proposition is similar to that of Theorem \ref{thm:onetoone}
in its spirit but is much simpler than it. It is enough to prove the map
$\kappa_a = \pi_\Delta \circ (\pi_2^{\Delta})^{-1}: S^2 \to S^2$ is one-one since
the map $(\pi_2^{\Delta})^{-1}: S^2 \to \Graph \phi_{F_i,a}^1$ is bijective.
But the map $\kappa_a$ is given by
$$
\kappa_a(y) = \frac{1}{2}(y +\phi_{F_i,a}^1(y))
$$
in the affine chart.  A straightforward computation shows
$$
d\kappa_a(y) =
\begin{cases}\frac{1}{2}(Id + d\phi_{F_i}^1(\frac{y}{a})) \quad & \text{for $y$ with } |y| \leq a(1-\eta) \\
Id \quad & \text{otherwise}
\end{cases}
$$
Since $\kappa_1$ is an orientation-preserving diffeomorphism and $S^2$ is compact, there exists $\delta> 0$ such that
$$
\det(d\kappa_1(y)) > \delta > 0
$$
for all $y \in S^2$.
From the expression of $d\kappa_a(y)$, it follows $d\kappa_a(y) = d\kappa_1(\frac{y}{a})$ and hence
\be\label{eq:dkappaay}
\det(d\kappa_a(y)) = \det\left(d\kappa_1(\frac{y}{a})\right) > \delta > 0
\ee
for all $a \in [0,1]$ and $y \in S^2$. This implies $\kappa_a: S^2 \to S^2$ is an immersion and
so a covering map of degree 1. Therefore it must be a one-one map.
\end{proof}

We now consider 3-chain $\Xi_i$ parameterized by the map
$$
\Xi_i: [0,1] \times M \to T^*\Delta
$$
defined by
\be\label{eq:Xii}
\Xi_i(a,x) = \begin{cases} a\, \left(\phi_i(x), x\right) \quad & x \in D^2_+ \\
(x,x) \quad & x \in S^2 \setminus D^2_+
\end{cases}
\ee
We note that this chain defines the same chain as the trace $\Tr_\CG(\phi_i)$ of
the Alexander isotopy of the 2-disc extended to $S^2$ by the identity: it is by definition parameterized by
$$
\Tr_\CG(\phi_i)(a,y) = \begin{cases} \left(a \phi_i(\frac{y}{a}), y \right) \quad & |y| \leq a \leq 1 \\
(y,y) \quad & y \in S^2 \setminus D^2_+(a)
\end{cases}
$$
\begin{rem}\label{rem:upshot} The way \eqref{eq:Xii} of parameterizing
the trace of the Alexander isotopy as a Lagrangian chain associated,
\emph{which is smooth everywhere including $a = 0$},
is crucial for us to establish the $C^0$-convergence property of $a$-Hamiltonian
of the Alexander isotopy in the context of graphical case. See
Corollary \ref{cor:Giconti} later in the next section. This aspect of
Alexander isotopy seems to be something new which,  as far as we know, has not been
utilized in the literature before.
\end{rem}

We will fix $i$ and just write $\phi = \phi_i$ until we need to vary $i$.
Using the coordinate change \eqref{eq:qptox}, we have the equalities
\be\label{eq:qapa}
{\bf q}^a = a\, \frac{\phi(x) + x}{2}, \, {\bf p}^a = -a\, j(\phi(x) - x)
\ee
on $D^2_+$ which extend to the $S^2 \setminus D^2_+$ by the identity map.
If we parameterize $\Graph \phi$ by ${\bf q}(1,x) = \frac{\phi(x) + x}{2}$
and ${\bf p}(1,x) =  - j(\phi(x) - x)$ (using the canonical identification of
$S^2 \cong \Delta_{S^2}$), then for each given $0 \leq a \leq 1$,
we may use Proposition \ref{prop:onetoonekappa} to parameterize
the chain $\Xi(a, \cdot)$
$$
{\bf p}(a,x) = d_{\bf q}{\bf g}(a,{\bf q}(a,x)), \, {\bf q}(a,x) = a {\bf q}(1,x)
$$
for $(a,x) \in [0,1] \times S^2$ for some continuous function ${\bf g} = {\bf g}(a,{\bf q})$
which is smooth on $(0,1] \times S^2$. More precisely, if we denote $g_a({\bf q}) = {\bf g}(a,{\bf q})$,
the second equation becomes
$$
{\bf p}_1(a,x) = \frac{\del g_a}{\del {\bf q}_1}({\bf q}(a,x)),
\, {\bf p}_2(a,x) = \frac{\del g_a}{\del {\bf q}_2}({\bf q}(a,x))
$$
on $D^2_+$. Substituting ${\bf q}(a,x) = a {\bf q}(1,x)$ thereinto, we obtain
$$
{\bf p}_1(a,x) = \frac{\del g_a}{\del {\bf q}_1}(a {\bf q}(1,x)),
\, {\bf p}_2(a,x) = \frac{\del g_a}{\del {\bf q}_2}(a {\bf q}(1,x)).
$$
Substituting the last three relations into the second equation of \eqref{eq:qapa}, we obtain
\be\label{eq:dgadqj}
\frac{\del g_a}{\del {\bf q}_j}(a {\bf q}(1,x)) = a \frac{\del g_1}{\del {\bf q}_j}({\bf q}(1,x)),
\quad j = 1, \, 2.
\ee
If we define a function $\widetilde g_a$ by $\widetilde g_a({\bf q}) = g_a(a\, {\bf q})$, then $\widetilde g_1 = g_1$ and
\be\label{eq:deltildega}
\frac{\del \widetilde g_a}{\del {\bf q}_j}({\bf q}) = a \frac{\del g_a}{\del {\bf q}_j}(a\, {\bf q}),
\quad j = 1, \, 2.
\ee
By the graphicality of $\phi$, it follows that the map
$$
x \mapsto {\bf q}(1,x) = \left(\frac{\phi(x) + x}{2}, \frac{\phi(x) + x}{2}\right)
$$
is a bijective map. Therefore by setting ${\bf q} = {\bf q}(1,x)$ and varying $x$,
we derive, from \eqref{eq:dgadqj} and \eqref{eq:deltildega},
$$
\frac{\del \widetilde g_a}{\del {\bf q}_j}({\bf q}) = a^2 \frac{\del \widetilde g_1}{\del {\bf q}_j}({\bf q}),
\quad j = 1, \, 2
$$
for all ${\bf q} \in D^2_+$.
This, combined with vanishing of $\widetilde g_a$ and $\widetilde g_1$ outside
$D^2_+$, implies
$$
\widetilde g_a = a^2 \widetilde g_1, \quad \text{ i.e,}\quad g_a(a \cdot) = a^2 g_1.
$$
\begin{prop}\label{prop:dgada} Let $g_a: \Delta_{S^2} \to \R$ be the function such that
$g_a \equiv 0$ outside $\Delta^{D^2_+}$ and $\Image dg_a$ represents
the 2-chain $\Xi(a,\cdot)$.  Denote ${\bf g}(a,{\bf q}) = g_a({\bf q})$. Then
\be\label{eq:dgada}
\frac{\del {\bf g}}{\del a}(a,{\bf q}) =
2a g_1\left(\frac{{\bf q}}{a}\right) -\frac{1}{a} dg_a ({\bf q})\cdot {\bf q}
\ee
for all ${\bf q} \in \Delta_{S^2}$ and for $0 < a \leq 1$.
\end{prop}
\begin{proof} We rewrite the identity $g_a(a{\bf q}) = a^2 g_1({\bf q})$ on $\Delta_{D^2_+}$
as ${\bf g}(a,a {\bf q}) = a^2 g_1({\bf q})$ on $D^2_+$.
By differentiating the latter identity with respect to $a$,  we obtain
$$
\frac{\del {\bf g}}{\del a}(a,a{\bf q}) + dg_a (a{\bf q})\cdot {\bf q} = 2a g_1({\bf q}).
$$
By setting $\widetilde{\bf q} = a{\bf q}$ for ${\bf q} \in \Delta_{D^2_+}$, we can rewrite
the equation as
$$
\frac{\del {\bf g}}{\del a}(a,\widetilde {\bf q})
+ \frac{1}{a} dg_a (\widetilde {\bf q})\cdot \widetilde {\bf q} = 2a g_1\left(\frac{\widetilde {\bf q}}{a}\right).
$$
By rewriting the variable $\widetilde {\bf q}$ by ${\bf q}$,
this proves  the equality \eqref{eq:dgada} on $\Delta_{D^2_+}$. It also
obviously holds on $\Delta_{S^2} \setminus \Delta_{D^2_+}$ since both sides vanish.
This finishes the proof.
\end{proof}

The upshot of this proposition is that the right hand side of \eqref{eq:dgada}
changes continuously with respect to the $C^0$-topology for the set of graphical $\phi$'s.
This will play a fundamental role in the proof of Theorem \ref{thm:weak-graphical}
later in the next section.

\begin{rem}
Strictly speaking, we should use the modified Alexander isotopy as given in
the proof of Theorem \ref{thm:alexander}.
We fix a sequence of weakly $\Psi$-graphical approximating sequence $F_i$ and
its Alexander isotopy $\Lambda_i = \Lambda_{i,\e_i}$ defined as in \eqref{eq:Lambdaiei}.
We also denote by $K_i = K_i(a,t,x)$ the unique associated $a$-Hamiltonian
supported in $\Int D^2$ chosen as in Lemma \ref{lem:choiceK}, and denote $G_i(a,x) = K_i(a,1,x)$.
Then we obtain $\Cal(G_i^a) = (\chi_i(a)^4 - \e_i^4) \Cal(G_i)$ from Lemma \ref{lem:Callambdaa}.
As $i \to \infty$, the right hand side converges to
$a^4 \overline{\Cal}(F)$ with $F = \lim_{i \to \infty} F_i$ since $\Cal(G_i)
= \Cal(F_i)$ from Lemma \ref{lem:Callambdaa}.
\end{rem}

Having said this remark, we will just use the standard Alexander isotopy given in Lemma
\ref{lem:Callambdaa} ignoring the fact that it is not smooth at $a = 0$.
All of our arguments can be justified using the above modified Alexander isotopy
and taking the limit.

\section{Vanishing of basic phase function for the graphical case on $D^2$}
\label{sec:vanishing}

Now we restrict to the context of Theorem \ref{thm:weak-graphical}.
Let $F$ be a topological Hamiltonian
generating a topological Hamiltonian loop $\phi_F$ on the
2-disc $D^2$ with $\supp F \subset \Int D^2$.
We consider an approximating sequence $H_i$ and $F_i = H_i(1)$
with $\supp F_i \subset \Int D^2$. We embed $D^2$ into $S^2$ as the upper
hemisphere and then extend $F_i$ canonically to whole $S^2$ by zero, and
consider the graphs $\Graph \phi_{F_i}^1$ in $S^2 \times S^2$.
Note $\supp \phi_{\mathbb F_i} \subset D^2_+ \times D^2_+$ and hence
$$
\Graph \phi_{F_i}^t \bigcap \Delta \supset \Delta_{D^2_-} \bigcup \Delta_{D^2_+ \setminus D^2_+(1-\delta)}
$$
for some $\delta > 0$ for all $t \in [0,1]$ independent of sufficiently large $i$'s
depending only on $F$, \emph{provided $\overline d(\phi_F^1,id) \leq \frac{\delta}{2}$}.
We fix the given topological Hamiltonian loop $\phi_F$ and fix such $\delta > 0$.

Then we consider the normalization $\underline{F_i}$ of $F_i$ on $S^2$ and define Hamiltonian
$$
\underline{\mathbb F_i} (t,{\bf x}): = \chi({\bf x})\, \underline{F_i}(t,x), \quad {\bf x} = (x,y)
$$
on $T^*\Delta$ with a slight abuse of notation for $\underline{\mathbb F_i}$.

\begin{thm}\label{thm:graphicalloop}
Conjecture \ref{conj:fvanishing} holds for any weakly graphical topological
Hamiltonian loop on $S^2$ arising as above.
\end{thm}

An immediate corollary of Theorem \ref{thm:average=Cal} and \ref{thm:graphicalloop}
is the following vanishing result of Calabi invariant.

\begin{cor} Suppose $\lambda = \phi_F$ be an engulfed topological Hamiltonian
loop as in Theorem \ref{thm:average=Cal}. Assume $\lambda $ is weakly graphical.
Then $\overline{\Cal}^{path}(\lambda) = 0$.
\end{cor}

The remaining section will be occupied by the proof of Theorem \ref{thm:graphicalloop}.
Let $F$ be a graphical topological Hamiltonian loop and $F_i$ be an approximating sequence
that is $\Psi$-graphical for a Darboux-Weinstein chart $\Psi$.

\begin{proof}[Proof of Theorem \ref{thm:graphicalloop}]
Let $F$ be the Hamiltonian associated to the topological Hamiltonian loop
on $S^2$ arising from the compactly supported Hamiltonian $F$ on $D^2$ that is weakly $\Psi$-graphical.
We fix a sequence of weakly $\Psi$-graphical approximating sequence $F_i$ and
its Alexander isotopy $\Lambda_i = \Lambda_{i,\e_i}$ defined as in \eqref{eq:Lambdaiei}.
(As we mentioned at the end of the last section, we will just use the standard
Alexander isotopy given in Lemma \ref{lem:Callambdaa} below for the simplicity of
exposition.)
We also denote $K_i = K_i(a,t,x)$ the unique associated $a$-Hamiltonian
supported in $\Int D^2$ chosen as in Lemma \ref{lem:choiceK}, and denote $G_i(a,x) = K_i(a,1,x)$.
In particular, $G_i(0, \cdot) = 0$. Recall $\Lambda_i(0,t) = id$ for all $t \in [0,1]$.

We denote $H_i(a)$ the $t$-Hamiltonian defined by $H_i(a)(t,x) = H_i(a,t,x)$.
By the weak $\Psi$-graphicality of $H_i(a)$ from Proposition \ref{prop:onetoonekappa}
for $a \in [0,1]$, ${G_i^a}$ is $\Psi$-graphical
where $G_i^a = G_i^a(s,x) = a\, G_i(as,x)$. We recall from Proposition \ref{prop:homotopybasic}
$\phi_{H_i(a)}^1 = \phi_{G_i}^a$ and
$$
f_{\mathbb H_i(a)} = f_{\mathbb G_i^a}.
$$
By the $\Psi$-graphicality of $G_i^a$, the basic phase function $f_{\mathbb G_i^a}$ is
defined everywhere on $\Delta$ as a smooth (single-valued) function.

Then the function ${\bf f}_{\underline{\mathbb G_i}}: [0,1] \times \Delta \to \R$ defined by
${\bf f}_{\underline{\mathbb G_i}}(a,{\bf q}) = f_{\underline{\mathbb G_i^a}}({\bf q})$ satisfies
the  Hamilton-Jacobi equation
\be\label{eq:ham-jacobi}
\frac{\del {\bf f}_{\underline{\mathbb G_i}}}{\del a}(a,{\bf q}) +
\underline{\mathbb G_i}\left(a, d_{\bf q}{\bf f}_{\underline{\mathbb G_i}}(a,{\bf q})\right) = 0.
\ee
We postpone the derivation of this equation till Appendix. We note
$d_{\bf q}{\bf f}_{\underline{\mathbb G_i}}(a,{\bf q}) = df_{\underline{\mathbb G_i^a}}({\bf q})$
by definition of $f_{\underline{\mathbb G_i^a}}$.

Now we consider the mean-normalized Hamiltonian $\underline{G_i^a}$ and its associated
Hamiltonian $\underline{{\mathbb G}_i^a}$ on $T^*\Delta$.

\begin{lem} Let ${\bf g}_i$ be the function associated to $\phi = \phi_{F_i}^1$ as
defined in Proposition \ref{prop:dgada}. Consider the $a$-Hamiltonian $K_i$ associated to
the Alexander isotopy. Then
\be\label{eq:faq}
{\bf f}_{\underline{\mathbb G_i}}(a,{\bf q}) = {\bf g}_i(a,{\bf q}) + a^4 \frac{\Cal(F_i)}{\vol(S^2)}.
\ee
\end{lem}
\begin{proof} We first recall $\Cal(F_i) = \Cal(G_i)$ (Lemma \ref{lem:CalKH}).

Since $f_{\mathbb G_i^a}$ and $g_{i,a}( := {\bf g}_i(a,\cdot))$
satisfy $\Image df_{\mathbb G_i^a} = \Image dg_{i,a}$ with $f_{\mathbb G_i^a} = g_{i,a} \equiv 0$
on $\Delta_{S^2} \setminus \Delta_{D^2_+}$,
$f_{\mathbb G_i^a} = g_{i,a}$ everywhere on $\Delta_{S^2}$. Then applying
Theorem \ref{thm:disccase}, we have finished the proof.
\end{proof}

\begin{cor}\label{cor:Giconti} We have
\be\label{eq:Giaunderline}
\underline{\mathbb G_i}\left(a, d_{\bf q}{\bf f}_{\underline{\mathbb G_i}}(a,{\bf q})\right)
= - 2a f_{\mathbb F_i}\left(\frac{{\bf q}}{a}\right) + \frac{1}{a} df_{\mathbb H_i(a)} ({\bf q})\cdot {\bf q}
-  4 a^3 \frac{\Cal(F_i)}{\vol(S^2)}.
\ee
In particular, the function
$(a,{\bf q}) \mapsto \underline{\mathbb G_i}(a, d_{\bf q}{\bf f}_{\underline{\mathbb G_i}}(a,{\bf q}))$
uniformly converges to a continuous function
$$
- 2a f_{\mathbb F}\left(\frac{{\bf q}}{a}\right) - 4a^3 \frac{\Cal(F)}{\vol(S^2)}.
$$
\end{cor}
\begin{proof} By differentiating \eqref{eq:faq}, and comparing \eqref{eq:ham-jacobi} and \eqref{eq:dgada},
we immediately obtain the first statement.

For the second, we rewrite
$$
\left|\frac{1}{a} df_{\mathbb H_i(a)} ({\bf q})\cdot {\bf q}\right| =
\left|df_{\mathbb H_i(a)} ({\bf q})\cdot \frac{\bf q}{a}\right| \leq |df_{\mathbb H_i(a)} ({\bf q})|\left|\frac{\bf q}{a}\right|.
$$
Then we recall $\supp f_{\mathbb H(a)} \subset D^2_+(a)$ (see \eqref{eq:Devlambdaa})
from which we derive
$$
|df_{\mathbb H_i(a)} ({\bf q})|\left|\frac{\bf q}{a}\right| \leq \max_{\bf q}|df_{\mathbb H_i(a)} ({\bf q})| \to 0
$$
as $i \to \infty$. Once this is established, we derive the second statement by taking
the limit of \eqref{eq:Giaunderline}.
\end{proof}

The explicit expression of the right hand side of \eqref{eq:Giaunderline} will not play any role later
in the present paper but only the conclusion that \emph{the function
$(a,{\bf q}) \mapsto \underline{\mathbb G_i}(a, d_{\bf q}{\bf f}_{\underline{\mathbb G_i}}(a,{\bf q}))$
uniformly converges to a continuous function} will do later.

We also derive from this corollary that there exists some constant $C> 0$ independent of $i$'s such that
\be\label{eq:uniformbound}
\left|\underline{\mathbb G_i}\left(a, d_{\bf q}{\bf f}_{\underline{\mathbb G_i}}(a,{\bf q})\right)\right|
\leq C
\ee
for all $a, \, {\bf q}$ and for all sufficiently large $i$'s. (In fact, we have
$$
\left|\underline{\mathbb G_i}\left(a, d_{\bf q}{\bf f}_{\underline{\mathbb G_i}}(a,{\bf q})\right)\right|
\leq 2a \|f_{\mathbb F_i}\|_{C^0} + \|df_{\mathbb H_i(a)}\|_{C^0} + 4 a^3\frac{|\Cal(F_i)|}{\vol(S^2)} \to 
4 a^3\frac{|\Cal(F)|}{\vol(S^2)}
$$
as $i \to \infty$.)

We now consider the integrals
$$
I_{G_i}(a) : = \int_\Delta f_{\underline{\mathbb G_i^a}}\, \pi_2^*\omega.
$$
Then $I_{G_i}(0) = 0$ and
$$
I_{G_i}'(a) = \int_\Delta \frac{\del {\bf f}_{\underline{\mathbb G_i}}}{\del a}\, \pi_2^*\omega
= \int_\Delta \underline{\mathbb G_i}\left(a,df_{\underline{\mathbb G_i^a}}({\bf q})\right)\, \pi_2^*\omega.
$$
and so
$$
I_{G_i}(1) =  \int_0^1 \int_\Delta \underline{\mathbb G_i}\left(a,df_{\underline{\mathbb G_i^a}}({\bf q})\right)
\, \pi_2^*\omega\, da.
$$
By the identity
$$
\phi_{\underline{\mathbb G_i}}^a(o_\Delta) = \Graph \phi_{\underline{G_i}}^a = \Image df_{\underline{\mathbb G_i^a}}
$$
and the bijectivity of the projection,
$$
\pi_2: \phi_{\underline{\mathbb G_i}}^a(o_\Delta) \to S^2
$$
we can write $\pi_1(df_{\underline{\mathbb G_i^a}}({\bf q})) = \phi_{\underline{G_i}}^a(y({\bf q}))$
for the unique $y({\bf q})$ satisfying
$$
\pi_2(df_{\underline{\mathbb G_i^a}}({\bf q})) = y({\bf q})
$$
for each given ${\bf q}$. We denote by
$\pi_1^{\phi_{\underline{\mathbb G_i}}^a(o_\Delta);\Delta}: \phi_{\underline{\mathbb G_i}}^a(o_\Delta) \to \Delta$
the projection of $\phi_{\underline{\mathbb G_i}}^a(o_\Delta)$ to $\Delta$ along $\pi_1$-direction.

Consider the sequence of maps
$\iota_{\phi_{\underline{\mathbb G_i}}^a}: \Delta \to \Delta$ defined by
$$
\iota_{\phi_{\underline{\mathbb G_i}}^a}({\bf q}) = \pi_1^{\phi_{\underline{\mathbb G_i^a}(\Delta);\Delta}} \circ \sigma_{\underline{\mathbb G_i^a}}
$$
where $\pi_1^{\phi_{\underline{\mathbb G_i^a}(\Delta);\Delta}}$ is the projection of
$\phi_{\underline{\mathbb G_i}}^a(\Delta)$ onto $\Delta$ along the $\pi_1$-direction.

Then we obtain the following from the graphicality of $\phi_{\underline{\mathbb G_i}}^a$.

\begin{lem}\label{lem:convergetoid}
Each element of the sequence is a diffeomorphism and that
the sequence $\iota_{\phi_{\underline{\mathbb G_i}}^a}$ uniformly converges to
the identity map $id_{\Delta}$ over $a \in [0,1]$ as $i \to \infty$.
\end{lem}

Now we recall the following well-known fact whose proof
follows from  a straightforward $3$ $\epsilon$ argument
and the weak continuity of the pushforward operation of measures under the $C^0$-topology of
continuous maps.

\begin{lem}\label{lem:weak-continuity} Let $X$ be a compact topological space. Denote by
$\mu$ a finite measure on $X$, by $f$ a continuous real-valued function.
Consider a sequence $f_i$ of continuous function
uniformly converging to $f$ and $\iota_i$ a sequence of continuous maps uniformly converging to
the identity map. Then
$$
\lim_{i \to \infty} \int_X f_i\, \mu  = \lim_{i \to \infty} \int_X (f_i \circ \iota_i) \, \mu \left(= \int_X f\, \mu\right).
$$
\end{lem}

Therefore applying this lemma to the current context of
$$
X = \Delta, \, \mu = \mu_\omega, \, \iota = (\iota_{\phi_{\underline{\mathbb G_i}}^a})^{-1}, \,
f_i = \underline{\mathbb G_i}\left(a, df_{\underline{\mathbb G_i^a}}(\cdot)\right)
$$
and combining Corollary \ref{cor:Giconti}, Lemma \ref{lem:convergetoid} and
\eqref{eq:uniformbound}, we derive
$$
\lim_{i \to \infty}
\int_\Delta \underline{\mathbb G_i}(a, df_{\underline{\mathbb G_i^a}}({\bf q}))\, \mu_\omega
=  \lim_{i\to \infty}
\int_\Delta \underline{\mathbb G_i}\left(a, df_{\underline{\mathbb G_i^a}}
(\iota_{\phi^a_{\underline{\mathbb G_i}}}^{-1}({\bf q}))\right)\, \mu_\omega.
$$
But for ${\bf q} = (y,y)$, recalling $\mathbb G_i = \pi_1^*G_i$ on
$\Graph \phi_{\mathbb G_i}^a = \Image df_{\underline{\mathbb G_i^a}}$,
we derive
\beastar
\underline{\mathbb G_i}\left(a, df_{\underline{\mathbb G_i^a}}
((\iota_{\phi_{\underline{\mathbb G_i}}^a})^{-1}({\bf q})\right)
& = & \underline{G_i}\left(a, \pi_1(df_{\underline{\mathbb G_i^a}}
((\iota_{\phi_{\underline{\mathbb G_i}}^a})^{-1}{\bf q})\right)\\
& = & \underline{\mathbb G_i}\left(a, \pi_1^{\phi_{\underline{\mathbb G_i}}^a(o_\Delta);\Delta}
\circ \sigma_{\underline{\mathbb G_i}^a}
((\iota_{\phi_{\underline{\mathbb G_i}}^a})^{-1}{\bf q})\right)\\
& = & \underline{\mathbb G_i}(a, {\bf q}) = \pi_1^*\underline{G_i}(a,{\bf q}) = \underline{G_i}(a,y).
\eeastar
For the second equality above, we also use the obvious identity
$$
\pi_1^{\phi_{\underline{\mathbb G_i}}^a(o_\Delta)}
= \pi_1^{\Delta} \circ \pi_1^{\phi_{\underline{\mathbb G_i}}^a(o_\Delta);\Delta}.
$$
Here for any given subset $L \subset S^2 \times S^2$,
we denote by $\pi_1^L: L \to S^2$ the restriction of $\pi_1:S^2 \times S^2 \to S^2$ to $L$.

Combining these, we evaluate the integral and derive
$$
\lim_{i\to \infty} \int_\Delta \underline{\mathbb G_i}\left(a, df_{\underline{\mathbb G_i^a}}({\bf q})\right)\, \pi_2^*\omega
= \lim_{i \to \infty} \int_{\Delta} \underline{\mathbb G_i}(a, {\bf q})\, \pi_2^*\omega
= \lim_{i \to \infty} \int_{S^2} \underline{G_i}(a,y)\, \omega = 0
$$
where the last vanishing occurs by the mean-normalization condition of $\underline{G_i}$.
This proves $\lim_{i\to \infty} I_{G_i}(1) = 0$ in particular.

But we have $f_{\underline{\mathbb G_i}}= f_{\underline{\mathbb H_i(1)}} (= f_{\underline{\mathbb F_i}})$
by Proposition \ref{prop:homotopybasic} and in particular
$$
f_{\underline{\mathbb G_i}}= f_{\underline{\mathbb F_i}} \to  f_{\underline{\mathbb F}}
$$
uniformly. Combining the above discussion, we have proved
$$
\int_\Delta f_{\underline{\mathbb F}}\, \pi_2^*\omega  = \lim_{i\to \infty}
\int_\Delta f_{\underline{\mathbb F_i}}\, \pi_2^*\omega = \lim_{i \to
\infty}\int_\Delta f_{\underline{\mathbb G_i}}\, \pi_2^*\omega =0.
$$
This finishes the proof.
\end{proof}

\begin{rem}
 We would like to point out that while the average of $\underline{G_i}$ vanishes and
$\phi_{G_i}^s \to id$ uniformly over $s \in [0,1]$, unlike the $t$-Hamiltonian
$F_i$ which converges in hamiltonian topology, there is no a priori control of
the $C^0$ behavior of the $s$-Hamiltonian $G_i$ itself in general according to the definition of
approximation sequence $H_i$ of the hameotopy in Definition \ref{defn:hameotopy}.
(See \eqref{eq:FLambda} for the explicit form of $G_i(s,\cdot) = K_i(s,1, \cdot)$ in the case of
Alexander isotopy, which evidently involves taking the derivative the $t$-Hamiltonian.)

In this regard, the above proof strongly relies on the graphicality of the topological
Hamiltonian, or more precisely on the graphicality of its approximation sequence.
Without this graphicality, one has to deal with emergence of the caustics of the
projection $\pi_\Delta : \Graph \phi_{\underline{F_i}} \to \Delta$ or equivalently the nondifferentiability
locus of the basic phase function $f_{\underline{\mathbb F_i}}$.
Here seems to enter the piecewise smooth Hamiltonian geometry of
Lagrangian chains. We will elaborate this aspect elsewhere.
\end{rem}

\appendix

\section{Homotopy invariance of basic phase function}
\label{sec:homotopybasic}

In this section, we prove some homotopy invariance property of
the basic phase function.

Let $\Lambda = \{\phi_{H(s)}^t\}$ be a smooth two-parameter family satisfying
$H \equiv 0$ on a neighborhood of $B$ by definition of $\CP^{ham}(Symp_U(M,\omega),id)$ with $U = M \setminus B$.
We denote by $K = K(s,t,x)$ a $s$-Hamiltonian of the 2-parameter family
$\Lambda = \left\{\phi_{H(s)}^t\right\}$ with $K(s,0, \cdot) \equiv 0$:
The latter choice is possible we have
the $s$-Hamiltonian flow $s \mapsto \phi_{H(s)}^0 \equiv id$ and so we can set
$K(s,0,\cdot) \equiv 0$. Recall that $K$ is determined uniquely modulo the
addition of constants depending on $s, \, t$.

We first prove a few lemmata.

\begin{lem}\label{lem:choiceK} Let $H$ and $K$ be as above. Suppose $B$ is connected and has nonempty interior.
Then we can choose the $s$-Hamiltonian $K$ so that
$K(s,t,\cdot) \equiv 0$ on a neighborhood of $B \subset M$ for all $s, \, t \in [0,1]$.
\end{lem}
\begin{proof} Let $H$ be as above and consider its associated $s$-Hamiltonian vector field
$Y$, i.e.,
$$
Y = \frac{\del \phi_{H(s)}^t}{\del s} \circ (\phi_{H(s)}^t)^{-1}.
$$
By definition, we have $Y \rfloor \omega$ is an exact one form that vanishes
on $B$. Considering the exact sequence
$$
\to H^0(B) \to H^1(M, B) \to H^1(M) \to
$$
and noting $H^0(B) = 0$, the map $H^1(M, B) \to H^1(M)$ has zero kernel.
This implies that
$Y \rfloor \omega = dK_{s,t}$ for some $K_{s,t}: M \to \R$ with $\supp K \subset M\setminus B$.
This finishes the proof.
\end{proof}

This in particular implies $\phi_{K^1} \in \CP^{ham}(Symp_U(M,\omega),id)$.
Next we have the following coincidence of the Calabi invariant.

\begin{lem}\label{lem:CalKH} Assume $H(0) = 0$ and let $K$ be chosen as in Lemma \ref{lem:choiceK}. Then
$$
\Cal_U(K^1) = \Cal_U(H(1))
$$
\end{lem}
\begin{proof} First note $\phi_{K^1}^1 = \phi_{H(1)}^1$.
Denote by $\Lambda(s,t) = \phi_{H(s)}^t$ the two-parameter family associated to $H$. Then
$$
\Lambda(0,t) \equiv id \equiv \Lambda(s,0)
$$
by the requirement $H(0,t,x) \equiv 0$. Therefore the Hamiltonian path $t \mapsto \phi_{H(1)}^t :=  \Lambda(1,t)$
is smoothly homotopic to
the path $s \mapsto \phi_{K^1}^s:= \Lambda(s,1)$ relative to the ends and hence we have
the lemma by the smooth homotopy invariance of $\Cal_U$: In fact, an explicit
homotopy $\Upsilon: [0,1]^2 \to Symp_U(M,\omega)$ between them is
given by the formula
$$
\Upsilon(s,t) = \begin{cases} \Lambda(t,1 + 2s(t-1)) \quad & \text{for } \, 0 \leq s \leq \frac{1}{2}\\
\Lambda\left(2(s-1/2) + 2t(1-s),t\right) \quad & \text{for } \, \frac{1}{2} \leq s \leq 1.
\end{cases}
$$
The map $\Upsilon$ satisfies
\beastar
\Upsilon(0,t) & = &  \Lambda(t,1) = \phi_{K^1}^t, \quad \Upsilon(1,t) = \phi_{H(1)}^t, \\
\Upsilon(s,0) & = &  id, \quad \Upsilon(s,1) = \Lambda(1,1) =\phi_{H(1)}^1 = \phi_F^1
\eeastar
and hence is the required homotopy relative to the ends.
\end{proof}

Now we prove homotopy invariance of the basic generating function and the basic phase functions.
\begin{prop}\label{prop:homotopybasic}
$\widetilde h_{\K^1} = \widetilde h_{\H(1)}$ and $f_{\K^1} = f_{\H(1)}$
\end{prop}
\begin{proof} We apply the first variation formula
\eqref{eq:1stvariation} to $z_{\K^1}^{\bf q}(s)$ and $z_{\H(1)}^{\bf q}(t)$
respectively, and obtain
\beastar
d\widetilde h_{\K^1}(v)& = &\langle \theta(\phi_{\K^1}^1({\bf q})), T\phi_{K^1}^1(v) \rangle\\
d\widetilde h_{\H^1}(v) & = & \langle \theta(\phi_{\H(1)}^1({\bf q})), T\phi_{\H(1)}^1(v) \rangle
\eeastar
for any given $v \in T_{\bf q} \Delta$.
Since $\phi_{\K^1}^1 = \phi_{\H(1)}^1$, we have proved
$d\widetilde h_{\K^1} = d\widetilde h_{\H(1)}$. On the other hand,
for any point ${\bf q} \in \Delta_B$, $\H \equiv 0 \equiv \K^1$ on a neighborhood of ${\bf q}$ in $T^*\Delta$
and so both $z_{\K^1}^{\bf q}$ and $z_{\H(1)}^{\bf q}$
are constant.
Therefore the values of both $\widetilde h_{\K^1}$ and $\widetilde h_{\H(1)}$
are zero at such a point ${\bf q} \in \Delta_B$. This finishes the proof of the first
equality.

For the proof of $\widetilde f_{\K^1} = \widetilde f_{\H(1)}$, the first
equality in particular implies that the sets of critical values of the action
functionals
$$
\CA^{cl}_{\K^1}, \, \CA^{cl}_{\H(1)}: \Omega(o_N,T_q^*M) \to \R
$$
coincide. Then standard homotopy argument used in the homotopy invariance of (in fact any type of)
the spectral invariant applies to prove
$\rho^{lag}(H, \{q\}) = f_{H}(q)$ for each $q \in N$ for general $H$.
This finishes the proof.
\end{proof}
Then combining Lemma \ref{lem:CalKH} and Proposition \ref{prop:homotopybasic}, we also derive
\be\label{eq:fK1=}
f_{\underline{\K}^1} =
f_{\K^1} + \frac{\Cal_U(K^1)}{\vol_\omega(M)} = f_{\H(1)} + \frac{\Cal_U(H(1))}{\vol_\omega(M)} =
f_{\underline{\H}(1)}
\ee

With this preparation, in the  proof of Theorem \ref{thm:graphicalloop} later,
we used ${\K}^1$ instead of $\H(1)$ in our proof. This is because we exploited
the fact that $\phi_{K^1}^s$ is $C^0$-small.

\section{Timewise basic phase function as a solution to Hamilton-Jacobi equation}
\label{append:HJ-equation}

In this section, we show that the space-time basic phase function ${\bf f}_H$
defined by
$$
{\bf f}_ H(t,q) = f_{H^t}(q)
$$
satisfies the Hamilton-Jacobi equation. More precise description of this statement is now in order.

Let $N$ be an arbitrary compact manifold without boundary and
let $H = H(t,x)$ be a time-dependent Hamiltonian defined on
the cotangent bundle $T^*N$ and $L = \phi_H^1(o_N)$ be the associated
Hamiltonian deformation of $o_N$. In this case, there is a canonical generating function
of $L$ associated to the Hamiltonian $H$ given as follows.

We first start with the discussion on the basic generating function. (We refer the readers to \cite{oh:lag-spectral} for
more detailed exposition on this.)
For any given time-dependent Hamiltonian
$H = H(t,x)$, the classical action functional on the space
$$
\CP(T^*N) : = C^\infty([0,1],T^*N)
$$
is defined by
\be\label{eq:CAH}
\CA^{cl}_H(\gamma) = \int \gamma^*\theta - \int_0^1 H(t,\gamma(t))\, dt.
\ee
We denote $L_H = \phi_{H}^1(o_N)$ and by $i_H:L_H \hookrightarrow T^*N$ the inclusion map.
For given $x \in L_H$, we define the Hamiltonian trajectory
$$
z_x^H(t) = \phi_H^t((\phi_H^1)^{-1}(x))
$$
which is one satisfying
$$
z_x^H(0) \in o_N, \quad z_x^H(1) = x.
$$
The function $h_H: L_H \to \R$, called the \emph{basic generating function} in
\cite{oh:lag-spectral}, is defined by
$$
h_H(x) = \CA_H(z_x^H).
$$
It satisfies $i_H^*\theta = dh_H$ on $L_H$, i.e., $h_H$ is a
canonical generating function of $L_H$ in that it satisfies
$$
i_H^*\theta = dh_H.
$$

Then we consider the parametric version of basic generating function \eqref{eq:tildehK1}
which is defined by
\be\label{eq:mathbbh}
{\bf h}_H(t,x) := h_{H^t}(x)
\ee
on $\Tr_{\phi_H}(o_N) := \bigcup_{t \in [0,1]} \{t\} \times \phi_H^t(o_N)$. A straightforward calculation leads to

\begin{prop}\label{prop:dhH}
Consider the map
$$
\Psi_H: [0,1] \times N \to T^*[0,1] \times T^*N \cong T^*([0,1] \times N)
$$
defined by the formula
\be\label{eq:suspension}
\Psi_H(t,q) = \left(t, - H(t,\phi_H^t(o_q)), \phi_H^t(o_q)\right)
\ee
where $o_q \in o_N$ associated to the point $q \in N$.
Then $\Psi_H$ is an exact Lagrangian embedding of $[0,1] \times N$.
Denote the associated exact Lagrangian submanifold by
$$
\widehat L: = \Image \Psi_H
$$
and by $i_{\widehat L}: \widehat L \to T^*([0,1] \times N)$ the inclusion map.
We also denote by $p: \widehat L \to [0,1] \times T^*N$ the restriction to $\widehat L$ of the
natural projection $T^*([0,1] \times N) \to [0,1] \times T^*N$.
Let $(t,a)$ be the canonical coordinate of $T^*[0,1]$.
Then the timewise basic generating functions $\widetilde{\bf h}_H, \, {\bf h}_H$ satisfy
\bea\label{eq:dhH-moving}
d\widetilde{\bf h}_H & = & \Psi_H^*(\theta + a\, dt) \nonumber\\
p^*d{\bf h}_H & = & i_{\widehat L}^*(\theta + a\, dt)
\eea
on $[0,1] \times N$ and on $\widehat L$ respectively.
In particular, ${\bf h}_H\circ p $ is a generating function of
the exact Lagrangian submanifold $\widehat L \subset T^*([0,1] \times N)$.
\end{prop}

As a function on the zero section $o_N \cong N$, not on $L_H$, the basic generating function $h_H$
is a multi-valued function. But the basic phase function ${\bf f}_H$, as a
timewise graph selector, satisfies the identity
\be\label{eq:hH=fHpi}
{\bf h}_H(t,x) = {\bf f}_H(t,\pi(x)).
\ee
In particular, substituting $x = d_q{\bf f}_H(t,q)$ into \eqref{eq:hH=fHpi}
and noting $\pi(d_q{\bf f}_H(t,q)) = q$, we obtain
$$
{\bf f}_H(t,q) ={\bf h}_H(t, d_x{\bf f}_H(t,q)) = {\bf h}_H \circ \sigma_H(t,q).
$$
Here $\sigma_H$ is the map defined by $\sigma_H(t,q) = df_{H^t}(q)$, which is
the timewise version of the definition \eqref{eq:sigmaHdfH} whose image is
contained in $\widehat L$.

Therefore
$$
d{\bf f}_H  =  d({\bf h}_H \circ \sigma_H) = \sigma_H^*(d{\bf h}_H)
= \sigma_H^* i_{\widehat L}^*(\theta + a\, dt)\\
 =  (i_{\widehat L} \circ \sigma_H)^*(\theta + a\, dt)
$$
on the smooth locus of ${\bf f}_H$ in $[0,1] \times N$.
But
$$
i_{\widehat L} \circ \sigma_H(t,q) = (t, -H(t, \sigma_H(t,q)), \sigma_H(t,q)).
$$
Therefore
$$
(i_{\widehat L} \circ \sigma_H)^*(\theta + a\, dt) = \sigma_{H}^*\theta - H(t, \sigma_H(t,q))\, dt
$$
and hence
$$
d{\bf f}_H = \sigma_{H}^*\theta - H(t, \sigma_H(t,q))\, dt.
$$
We also have $d\pi d\sigma_H(\frac{\del}{\del t}) = 0$ since $\pi \sigma_H(t,q) = q$ for all $t$.
This implies
$$
\sigma_H^*\theta(\frac{\del}{\del t}) = \sigma_H(t,q)\left(d\pi d\sigma_H(\frac{\del}{\del t})\right) =  0.
$$
In particular, we have derived
$$
\frac{\del {\bf f}_H}{\del t} =  - H(t, \sigma_H(t,q))
$$
(on the smooth locus $N \setminus \operatorname{Sing}(d{\bf f}_H)$).
This is equivalent to the Hamilton-Jacobi equation
$$
\frac{\del {\bf f}_H}{\del t}(t,q) + H(t,d_q{\bf f}_H(t,q)) = 0.
$$

By repeating the above discussion for the time $s$ instead of $t$ in our present context of $N = \Delta$
we have derived the Hamilton-Jacobi equation \eqref{eq:ham-jacobi}.

\end{document}